\documentclass[11pt,letterpaper]{article}

\usepackage{amssymb,amsthm,amsmath}
\usepackage[pdftex]{graphicx}
\usepackage{graphicx}
\usepackage[margin=1in]{geometry}
\usepackage{color}
\usepackage{array}
\usepackage{capt-of}
\usepackage{url}
\usepackage{authblk}

\newtheorem{theorem}{Theorem}[section]
\newtheorem{lemma}[theorem]{Lemma}
\newtheorem{proposition}[theorem]{Proposition}

\theoremstyle{definition}

\theoremstyle{remark}

\newcommand{\bbR}{\mathbb{R}}
\newcommand{\R}{\mathbb{R}}

\newcommand{\be}{\mathbf{e}}

\newcommand{\bu}{\mathbf{u}}
\newcommand{\bv}{\mathbf{v}}
\newcommand{\bw}{\mathbf{w}}
\newcommand{\bx}{\mathbf{x}}

\newcommand{\cN}{\mathcal{N}}

\newcommand{\tA}{\mathbf{A}}

\newcommand{\tP}{\mathbf{P}}

\newcommand{\tV}{\mathbf{V}}
\newcommand{\tX}{\mathbf{X}}
\newcommand{\tY}{\mathbf{Y}}
\newcommand{\tZ}{\mathbf{Z}}
\newcommand{\tU}{\mathbf{U}}

\newcommand{\bzero}{\mathbf{0}}
\newcommand{\bone}{\mathbf{1}}

\newcommand{\tr}{\mathrm{tr}}
\renewcommand{\vec}{\mathrm{vec}}

\makeatletter
\newcommand*{\rom}[1]{\expandafter\@slowromancap\romannumeral #1@}
\makeatother

\title{Exactness Conditions for Semidefinite Relaxations of the Quadratic Assignment Problem}
\author[1]{Junyu Chen \thanks{Email: chenjunyu@u.nus.edu}}
\author[1,2]{Yong Sheng Soh}
\affil[1]{Department of Mathematics, National University of Singapore, Singapore 119076}
\affil[2]{Institute of Operations Research and Analytics, National University of Singapore, Singapore 119076}

\begin{document}

\maketitle

\begin{abstract}
The Quadratic Assignment Problem (QAP) is an important discrete optimization instance that encompasses many well-known combinatorial optimization problems, and has applications in a wide range of areas such as logistics and computer vision.  The QAP, unfortunately, is NP-hard to solve.  To address this difficulty, a number of semidefinite relaxations of the QAP have been developed.  These techniques are known to be powerful in that they compute globally optimal solutions in many instances, and are often deployed as sub-routines within enumerative procedures for solving QAPs.  In this paper, we investigate the strength of these semidefinite relaxations.  Our main result is a deterministic set of conditions on the input matrices -- specified via linear inequalities -- under which these semidefinite relaxations are exact.  Our result is simple to state, in the hope that it serves as a foundation for subsequent studies on semidefinite relaxations of the QAP as well as other related combinatorial optimization problems.  As a corollary, we show that the semidefinite relaxation is exact under a perturbation model whereby the input matrices differ by a suitably small perturbation term.  One technical difficulty we encounter is that the set of dual feasible solutions is not closed.  To circumvent these difficulties, our analysis constructs a sequence of dual feasible solutions whose objective value approaches the primal objective, but never attains it.  
\end{abstract}

\emph{Keywords}: Graph matching, graph isomorphism problem, spectrahedral shadows

\section{Introduction}

Suppose there are $n$ facilities and a set of $n$ locations.  For each pair of locations we are given information about the {\em distance} between these locations, and for each pair of facilities we are given information about the amount of {\em flow} between these facilities.  The quadratic assignment problem seeks an assignment of facilities to locations so as to minimize the total flow, weighted by the distances.  

Concretely, let $A$ and $B$ be $n \times n$ dimensional real matrices and denote by $[n]$ the set $\{1,2,\ldots,n\}$.  The quadratic assignment problem (QAP) seeks the optimal permutation $\sigma : [n] \rightarrow [n]$ that minimizes the following sum
\begin{equation}
\underset{\sigma : [n] \rightarrow [n]}{\min} ~~ \sum_{1\leq i,j \leq n} A_{ij} B_{\sigma(i)\sigma(j)}.
\end{equation}
The matrix $A$ models the amount of flow between facilities, while the matrix $B$ models the distances between locations.  In what follows, we will only be concerned with the mathematical description of the QAP, and will make no further mention of the model interpretation.  To simplify some of the exposition later, we will make the additional assumptions that $A$ and $B$ are symmetric matrices.

The QAP can be expressed as an optimization over the set of permutation matrices.  Define $\text{Perm}(n) := \{ X : X^T X = XX^T = I, X_{ij} \in \{0,1\} \} \subset \R^{n\times n}$.  Then the QAP is equivalent to the following
\begin{equation} \label{eq:qap}
    \min_{X} ~ \tr (X A X^T B) \qquad \text{s.t.} \qquad X \in \text{Perm}(n).
\end{equation}
Note that the objective can be expressed as $\tr (X A X^T B) = \sum_{ijkl} A_{ij} B_{kl} X_{ik} X_{jl}$.  There is in fact a more general formulation of the QAP in which one replaces the terms $A_{ij} B_{kl}$ with a general cost tensor $C_{ijkl} \in \mathbb{R}^{n \times n \times n \times n}$.  Most of our results will apply to this generalization in a straightforward fashion.

The QAP is an important and widely studied optimization instance.  Part of the reason why it is important is because many combinatorial optimization problems can be modeled as instances of a QAP; prominent examples include the graph isomorphism problem, the sub-graph isomorphism problem, and shape correspondence problems \cite{shapecorr}.  
Unfortunately, the family of QAPs is NP-Hard \cite{qapnphard}.  In fact, this notion of hardness also extends to a practical sense in the sense that real life QAP instances are generally regarded as being very difficult to solve in practice \cite{qapsolution,reducibility,qapsurvey} -- a QAP instance of about 40-50 is generally considered to be state of the art \cite{admmqapsdp}.

\subsection{Semidefinite Relaxations}  

A prominent approach for solving QAPs involves the use of semidefinite relaxations.  Concretely, semidefinite programs (SDPs) are convex optimization instances in which one minimizes linear objectives over convex constraint sets specified as the intersection of the cone of positive semidefinite matrices (PSD) and an affine space.  The typical manner in which semidefinite relaxations are deployed in the context of QAPs is to expand the feasible region to one that can be described efficiently via semidefinite programming.  These provide lower bounds for the QAP that can be efficiently computed, and hence semidefinite relaxations can be used as a sub-routine within enumerative schemes for solving QAPs.

Semidefinite relaxations for QAPs are among the most powerful relaxations known in the sense that these methods tend to provide some of the best known bounds on certain test datasets.  The main downside of these methods, unfortunately, is that they remain prohibitively expensive to be practically feasible.  To this end, there has been a line of work focused on developing alternative weaker relaxations -- often by selectively omitting certain constraints -- that are easier to scale, but provide relatively strong relaxations on most instances.

The goal of this paper is to understand the {\em strength} of these semidefinite relaxations.  Our main result is to provide a simple condition that describe QAP instances for which these semidefinite relaxations are {\em exact}; that is, the optimal value obtained by these relaxations equals that of the QAP.  Our answer has important practical implications -- a by-product of our analyses is that we identify instances of the QAP that can be broadly considered as `easy instances' of the QAP as these are solvable using semidefinite relaxations.   

In a short while, we describe the semidefinite relaxation we focus on in this paper.  First, let $\text{DS}(n)$ be the set of doubly stochastic matrices
$$
\text{DS}(n) := \{ X : X \bone = \bone, X^T \bone = \bone, X \geq 0 \} \subset \R^{n\times n}.
$$
As a reminder, the Birkhoff-von Neumann theorem tell us that the set of doubly stochastic matrices is the convex hull of the set of permutation matrices $\text{DS}(n) = \mathrm{conv}(\text{Perm}(n))$.  In the above, $\bone = (1,\ldots,1)^T$ is the vector whose entries are all ones.  

Note that $ \tr (X A X^T B) = \langle X, B X A \rangle = \langle \vec (X), \vec (B X A) \rangle = \langle \vec (X), A \otimes B ~ \vec (X) \rangle$, where the symbol $\otimes$ is used to denote the Kronecker product.  As such, \eqref{eq:qap} is equivalent to 
\begin{equation} \label{eq:qap2}
\begin{aligned}
\min_{X,\tX} \qquad & \langle A \otimes B, \tX \rangle \\
  \text{s.t.} \qquad & \left(\begin{array}{cc}
    \tX & \text{vec}(X)\\
    \text{vec}(X)^T & 1
    \end{array}\right) \succeq 0,  \text{ is rank-one}\\
    & X \in \text{Perm} (n)
\end{aligned}.
\end{equation}
Here, we use $A\succeq B$ to denote $A-B$ is positive semidefinite (PSD).  We introduce the following notation for operators $\tX \in \mathbb{R}^{n^2 \times n^2}$; specifically, we denote
$$
\tX_{(i,j),(k,l)} \quad \text{where} \quad i,j,k,l \in [n]
$$
such that if $\tX = \vec(X) \vec(X)^T$ (that is, the rank-one constraint is active), then one has
$$
\tX_{(i,j),(k,l)} = X_{ij} X_{kl}.
$$
The formulation \eqref{eq:qap2} is an alternative formulation of the QAP in which the objective is {\em linear}.  The feasible region are those in which the matrix is rank-one, and these take the form
$$
\left \{ \left(\begin{array}{c}
\text{vec}(X)\\ 1
\end{array}\right) 
\left(\begin{array}{c}
\text{vec}(X)\\ 1
\end{array}\right)^T : X \in \text{Perm}(n) \right \}.
$$
Unfortunately, the convex hull of this set is {\em not} tractable to describe.

In this paper, we focus on the following SDP relaxation
\begin{equation}\label{eq:qap-sdp_temporarytag} 
\begin{aligned}
\min_{ \tX, X } \qquad & \langle A \otimes B, \tX \rangle \\
\text{s.t.} \qquad & \left(\begin{array}{cc}
    \tX & \text{vec}(X)\\
    \text{vec}(X)^T & 1
    \end{array}\right) \succeq 0 \\
    & \sum_{i,j} \tX_{(i,k),(j,l)} = 1, \sum_{i,j} \tX_{(k,i),(l,j)} = 1 \text{ for all } k, l \\
    & \tX_{(i,k),(j,k)} = 0, \tX_{(k,i),(k,j)} = 0  \text{ if } i \neq j \\
    & \sum_j \tX_{(i,j),(i,j)} = 1, \sum_j \tX_{(j,i),(j,i)} = 1 \text{ for all } i\\
    & \tX_{(i,j),(i,j)} = X_{ij} \\
    & \tX \ge 0 \\
    & X \in \text{DS}(n) 
\end{aligned}.
\end{equation}


The relaxation \eqref{eq:qap-sdp_temporarytag} is obtained by omitting the rank-one constraint in \eqref{eq:qap2}, and by subsequently adding a number of valid (in-)equalities that hold whenever $\tX = \vec(X) \vec(X)^T$, and $X$ is a permutation matrix.  First, the equality in the second row is valid because $ \sum_{i,k} \tX_{(i,j),(k,l)} = (\sum_{i} X_{ij}) (\sum_k X_{kl}) =  1$.  Second, the equalities in the third and fourth rows are consequences of orthogonality.  Specifically, if $X$ is orthogonal, then $\sum_k X_{ik} X_{jk} = \delta_{ij}$.  In the case where $i \neq j$, the additional constraint that $\tX \geq 0$ forces $\tX_{(i,k),(j,k)} = 0$, which are the equalities in the third row.  This forcing condition is also referred to the {\em Gangster operator}.  Last, the equality $\tX_{(i,j),(i,j)} = X_{ij}$ is valid because $X_{ij} \in \{0,1\}$ and hence $X_{ij}X_{ij}=X_{ij}$.  This constraint is sometimes referred to as the {\em arrow} condition. 

{\em Eliminating redundant constraints.}  The semidefinite relaxation \eqref{eq:qap-sdp_temporarytag} contains a number of redundant constraints that can be omitted.  First, the set of constraints $\sum_j \tX_{(i,j),(i,j)} = 1$ is redundant.  To see why, recall the constraint $\sum_{i,j} \tX_{(k,i),(k,j)} = 1$.  Among these terms in the sum, those of the form $\tX_{(k,i),(k,j)}$ are zero whenever $i \neq j$.  This implies $\sum_j \tX_{(i,j),(i,j)} = 1$.  (An identical set of arguments imply $\sum_j \tX_{(j,i),(j,i)} = 1$ is redundant.)  Second, the arrow condition $\tX_{(i,j),(i,j)} = X_{ij}$ is also redundant; however, this step is less straightforward to establish.  We prove this result in Appendix \ref{append:redundancy}. 
In our subsequent discussion, 
we omit these redundant constraints, and consider the following equivalent but simplified relaxation:

\begin{equation} \label{eq:qap-sdp}
\begin{aligned}
\min_{ \tX, X } \qquad & \langle A \otimes B, \tX \rangle \\
\text{s.t.} \qquad & \left(\begin{array}{cc}
    \tX & \text{vec}(X)\\
    \text{vec}(X)^T & 1
    \end{array}\right) \succeq 0 \\
    & \sum_{i,j} \tX_{(i,k),(j,l)} = 1, \sum_{i,j} \tX_{(k,i),(l,j)} = 1 \text{ for all } k, l \in [n]\\
    & \tX_{(i,k),(j,k)} = 0, \tX_{(k,i),(k,j)} = 0  \text{ if } i \neq j \\
    & \tX \ge 0 \\
    & X \in \text{DS}(n) 
\end{aligned}.
\end{equation}

\section{Our Contributions and Related Work}

\subsection{Our Contributions}  

The objective of this paper is to investigate the {\em strength} of the relaxation \eqref{eq:qap-sdp}.  Our main result is Theorem \ref{thm:exactness_indeterminate}, which is a set of deterministic conditions -- specified via linear inequalities -- on matrices $A$ and $B$ under which the relaxation \eqref{eq:qap-sdp} is {\em exact}.  The main utility of our result is that the conditions are simple to state, and as we suggest in Section \ref{sec:numerics}, characterizes a large proportion of the instances for which the relaxation \eqref{eq:qap-sdp} is exact.  Our hope is that Theorem \ref{thm:exactness_indeterminate} as well as the proof techniques introduced in this paper can serve as the foundation for subsequent work to prove exactness of semidefinite relaxations for specific applications of the QAP, such as the graph matching problem.

\subsection{Related Work}  

We briefly describe related work and our contributions in relation to these works.  

{\bf Exactness conditions for QCQPs.} The first body of work that closely relates to ours concerns those that study conditions under which semidefinite relaxations for certain computationally intractable problems are exact.  One important example is the family of quadratically constrained quadratic programs (QCQPs) -- these are problems where the objective are quadratic functions and where the constraints are specified by quadratic inequalities.  In fact, the QAP can be specified as an instance of a QCQP.  

To derive semidefinite relaxations for QCQPs, one applies a lift analogous to \eqref{eq:qap2} whereby one replaces quadratic variables $\bx\bx^T$ with the matrix $X$, imposes the following PSD constraint
$$
\left( \begin{array}{cc} X & \bx \\ \bx^T & 1 \end{array}\right) \succeq 0,
$$
and omit the rank-one condition.  This approach is frequently referred to as the {\em Shor relaxation} \cite{shorrelax}.  There are recent works that provide conditions under which these relaxations are exact for general QCQPs  \cite{burer2020exact,wang2022tightness}.  Unfortunately, these results are not directly applicable to studying semidefinite relaxations for the QAP such as \eqref{eq:qap-sdp} because the vanilla Shor relaxations of the QAP are far too weak to be useful.  Besides QCQPs, there are related works that provide exactness conditions for semidefinite relaxations for other computational problems such as matrix completion \cite{svt,exactmccandesfocm,exactmccandesacm}, phase retrieval \cite{candes2015phase}, group synchronization \cite{bandeira2018randomsychro,ling2022improvedsynchro,zhong2018nearoptsynchro}, data clustering \cite{relaxclusteringawasthi,certifiablekmeans,certifyspectralclustering,rujeerapaiboon2019size} and graph coloring \cite{lovaszthetagraphcoloring}.


{\bf Graph matching.}  Let $\mathcal{G}_A$ and $\mathcal{G}_B$ be two undirected graphs over $n$ vertices.  The graph matching problem concerns finding a mapping between these sets of vertices so as to maximize the overlap between edges.  Suppose one is specifically interested in maximizing for the {\em number} of overlapping edges.  Then the objective can be expressed as
\begin{equation} \label{eq:graphmatching_quadratic}
\min_{X \in \mathrm{Perm}(n)} \| X A - B X \|_F^2,
\end{equation}
where $A$ and $B$ are the (symmetric) adjacency matrices of $\mathcal{G}_A$ and $\mathcal{G}_B$ respectively.  By expanding the above objective and eliminating constants, the graph alignment instance can be rewritten as the following QAP instance
\begin{equation} \label{eq:qap_graphmatching}
\min_{X \in \mathrm{Perm}(n)} ~~ - \mathrm{tr}(X A X^T B).
\end{equation}
As compared to QAPs, graph alignment problems have a lot more structure -- namely, the matrices $A$ and $B$ must be binary.  Nevertheless, the graph alignment problem contains, for instance, the sub-graph isomorphism problem as a special case, which is known to be NP-hard.

There is a substantial body of work on graph matching problems, and in particular, approaches based on semidefinite relaxations of the QAP instance \eqref{eq:qap_graphmatching}.  Very frequently, these computational techniques are also applicable to solving more general QAP instances where the matrices $A$ and $B$ need not be binary, and as such we discuss the QAP and graph matching problems jointly.  More recently, the graph matching problem has been widely explored across various statistical models, with the primary goal of determining the noise levels at which an efficient algorithm can correctly identify the optimal underlying permutation \cite{spectralGMgaussian,spectralGMerdos,randomGMimproved,exactrandomGM}.  (On this note, we recommend the introduction in \cite{partialrecoveryhall} for an excellent treatment of these problems.)

{\bf Semidefinite relaxations for the QAP and the graph matching problem.} The SDP relaxation we study \eqref{eq:qap-sdp} is mathematically equivalent to the formulation suggested by Zhao et. al. \cite{zhao1998semidefinite} (it is also equivalent to Relaxation II in \cite{ling2024exactnessqapsdp}).  It is widely known to be a powerful relaxation -- unfortunately, the specific formulation is expensive to implement in practice.  To this end, a number of weaker relaxations that seek to drop certain constraints while mitigating the drop in performance have been suggested.  Some of these weaker relaxations include using a simple spectral relaxation \cite{leordeanu2005spectral}, relaxing the set of permutation matrices to the set of doubly stochastic matrices \cite{aflalo2015convex,lyzinski2015graph}, and alternative semidefinite relaxations \cite{bernard2018ds,bravo2018semidefinite,de2012improvedsdp,ds++,kezurer2015tight,povh2009copositive,zhao1998semidefinite}.

{\bf Exactness conditions.}  In \cite{aflalo2015convex}, the authors study a relaxation of the graph matching problem in which they replace the constraint set in \eqref{eq:graphmatching_quadratic} with the set of doubly stochastic matrices.  The authors show that this relaxation is exact if (i) the underlying graphs are isomorphic, and (ii) the adjacency matrix has no repeated eigenvalues and no eigenvector is orthogonal to the all-ones vector $\bone$.

A very recent paper that appeared as we were preparing our manuscript investigates extends these conditions within a signal corruption model \cite{ling2024exactnessqapsdp}.  Concretely, the work investigates the relaxation \eqref{eq:qap-sdp} and assumes that the matrices satisfy $B \approx - A + \Delta$, where $\Delta$ is suitably small.  (As a reminder, if $B=-A$, then the identity matrix $I$ is the optimal solution.)  The key result in \cite{ling2024exactnessqapsdp} is that the semidefinite relaxation \eqref{eq:qap-sdp} is exact and the identity matrix $I$ is the unique optimal solution if the matrix $A$ has distinct eigenvalues and the perturbation term $\Delta$ satisfies
\begin{equation} \label{eq:ling_result}
\min_{i\neq j} (\lambda_i - \lambda_j)^2 \min_i \langle \bu_i, \bone \rangle^2 \geq n (\|\Delta\| \|A \| + \| A \Delta\|_{\max}).
\end{equation}
Here, $\|\cdot\|$ denotes the spectral norm, $\|\cdot\|_{\max}$ denotes the maximum absolute value, while $\{\lambda_i,\bu_i\}$ are the eigenvalues and eigenvectors of the matrix $A$.  In particular, the result \eqref{eq:ling_result} can be viewed as a robust extension to the results in \cite{aflalo2015convex}.


{\bf Our work in relation to prior work.}  How does our work differ from these prior works?  At a high level, our objective is quite different from prior work.  Our aim is to characterize (to the extent we are able) the collection of matrices for which the semidefinite relaxation \eqref{eq:qap-sdp} is exact, which we believe is more ambitious than prior works that deal with specific generative models.  At a technical level, our result in Theorem \ref{thm:exactness_indeterminate} avoids the key assumptions in \cite{aflalo2015convex,ling2024exactnessqapsdp} -- namely, it does not require the eigenvalues to be distinct nor does it require eigenvectors to contain some component in the direction $\bone$.  This aspect is particularly relevant in certain applications such as graph matching problems as the input matrices may violate the spectral conditions imposed by these prior works.  The flipside is that our result only guarantees that the objective value of the relaxation is equal to the QAP objective -- we do not provide guarantees that the optimal solution is unique and is equal to a permutation.  In Section \ref{sec:geom} we discuss an example where the spectrum of $A$ contains repeated eigenvalues, where a subset of the eigenvectors are orthogonal to $\bone$, and where our result certifies exactness.  We show in this particular example that the solution is not unique.  We think there might be a connection between spectral conditions imposed by prior works and the semidefinite relaxation \eqref{eq:qap-sdp} having a unique optimal solution.  Finally, the proof technique of \cite{ling2024exactnessqapsdp} relies on constructing a dual feasible solution.  In this paper, we show that the set of dual solutions is in fact {\em not} closed, and in particular, there are matrices $A$ and $B$ that require the construction of a sequence of dual feasible solutions whose objective converges towards, but never attains the primal objective!  This phenomenon is peculiar to semidefinite programs in that the feasible regions of semidefinite programs -- these are known as {\em spectrahedral shadows} -- need not be closed \cite{Netzer:10}.  In particular, we believe that our proof techniques may be of broader interest especially to researchers studying semidefinite relaxations.

\section{Main Results}

Our main result is a collection of deterministic conditions on pairs of matrices $A$ and $B$ for which the semidefinite relaxation in \eqref{eq:qap-sdp_temporarytag} is exact.  For these conditions to be meaningful they have to (i) be relatively easy to describe, and (ii) they should capture the bulk of instances for which \eqref{eq:qap-sdp} is indeed exact.  To simplify the exposition, we only state the conditions under which the optimal solution to \eqref{eq:qap-sdp} is the tuple $(I, \vec(I) \vec(I)^T)$.  It is straightforward to extend these results via a suitable permutation of the indices to settings where the optimal solution is $(X, \vec(X) \vec(X)^T)$ for some $X \in \text{Perm}(n)$.

\begin{theorem}\label{thm:exactness_indeterminate}
Let $A, B \in \mathbb{R}^{n\times n}$ be symmetric.  Suppose there is a collection of vectors $\{\bu^{(ij)}\}_{i,j \in [n]}$, $\{\bv^{(kl)}\}_{k,l \in [n]} \subset \mathbb{R}^{n}$ satisfying
\begin{enumerate}
\item $u^{(ij)}_k+u^{(ij)}_l + v^{(kl)}_i+v^{(kl)}_j \leq A_{ij} B_{kl}$ for all ($i \neq j,k \neq l$) and all ($i=j,k=l$), 
\item $u^{(ij)}_i+u^{(ij)}_j + v^{(ij)}_i+v^{(ij)}_j = A_{ij} B_{ij}$ for all $i, j$, 
\item $\sum_{i,j}  u^{(ij)}_{\sigma(i)}+u^{(ij)}_{\sigma(j)} + v^{(ij)}_{\sigma^{-1}(i)}+v^{(ij)}_{\sigma^{-1}(j)} \geq \sum_{i,j} A_{ij} B_{ij}$ for all permutations $\sigma$.
\end{enumerate}
Then the tuple $(X,\tX) = (I, \vec(I) \vec(I)^T)$ is an optimal solution to \eqref{eq:qap-sdp}.
\end{theorem}

We prove Theorem \ref{thm:exactness_indeterminate} in Section \ref{sec:proof_exactness}.

We make a few remarks concerning Theorem \ref{thm:exactness_indeterminate}.  Our first remark is that, by combining Requirements 1 and 3, one obtains 
\begin{equation*}
\begin{aligned}
\sum_{ij} A_{ij} B_{\sigma(i)\sigma(j)} ~\geq~ & \sum_{ij} u^{(ij)}_{\sigma(i)}+u^{(ij)}_{\sigma(j)} + v^{(\sigma(i)\sigma(j))}_{i}+v^{(\sigma(i)\sigma(j))}_{j} \\
~=~ & \sum_{ij} u^{(ij)}_{\sigma(i)}+u^{(ij)}_{\sigma(j)} + v^{(ij)}_{\sigma^{-1}(i)}+v^{(ij)}_{\sigma^{-1}(j)} ~\geq~ \sum_{ij} A_{ij} B_{ij}
\end{aligned}
\end{equation*}
for all permutations $\sigma$.  This implies that $I$ is the optimal solution to the QAP \eqref{eq:qap}.  In this sense, the collection of vectors $\{\bu^{(ij)}\}_{i,j \in [n]}, \{\bv^{(kl)}\}_{k,l \in [n]} \subset \mathbb{R}^{n}$ serves as a certificate that $I$ is the optimal solution to the QAP; in particular, the semidefinite relaxation \eqref{eq:qap-sdp} seeks certificates of optimality of the form specified in Theorem \ref{thm:exactness_indeterminate} among all possible permutations.

Our second remark is that if one is able to find a set of vectors such that equality is attained in the Requirement 1 of Theorem \ref{thm:exactness_indeterminate}; that is
$$u^{(ij)}_k+u^{(ij)}_l + v^{(kl)}_i+v^{(kl)}_j = A_{ij} B_{kl},$$ 
for all ($i \neq j,k \neq l$) and all ($i=j,k=l$), then Requirement 3 is precisely equal to $I$ being an optimal solution to the QAP.  This is possible for $n=3$, for instance.

\begin{theorem} \label{thm:exactness_n3}
The SDP relaxation \eqref{eq:qap-sdp} is exact when $n=3$.
\end{theorem}

\begin{proof} [Proof of Theorem \ref{thm:exactness_n3}]
We assume without loss of generality that $I$ is an optimal solution to \eqref{eq:qap}. We choose
$$
\begin{aligned}
& u^{(12)}_1 = A_{12}(B_{12} + B_{13} - B_{23})/4, ~ v^{(12)}_1 = B_{12}(A_{12} + A_{13} - A_{23})/4 \\
& u^{(23)}_1 = A_{23}(B_{12} + B_{13} - B_{23})/4, ~ v^{(23)}_1 = B_{23}(A_{12} + A_{13} - A_{23})/4 \\
& u^{(31)}_1 = A_{31}(B_{13} + B_{13} - B_{23})/4, ~ v^{(31)}_1 = B_{31}(A_{12} + A_{13} - A_{23})/4. \\
\end{aligned}
$$
The other coordinates are defined similarly via a cyclic shift (of the subscripts in the brackets).  In addition, we choose $u^{(ii)}_j = A_{ii}B_{jj}/4$ and $v^{(ii)}_j = A_{jj}B_{ii}/4$ for all $i,j \in [n]$. 
 Then notice that Requirement 1 in Theorem \ref{thm:exactness_indeterminate} holds with equality.  Combining this with the assumption that $I$ is an optimal solution to \eqref{eq:qap} implies exactness (by our second remark). 
\end{proof}

We cannot do this in general.  This is because the vectors only provide $2 \binom{n}{2} \times n$ degrees of freedom whereas there are ${n\choose 2} \times {n\choose 2}$ equality conditions to satisfy.

Our second remark is that, if one chooses the entries of $u^{(ij)}_k$ and $v^{(kl)}_i$ such that $u^{(ij)}_k = u^{(ij)}_{k'}$ and $v^{(kl)}_i = v^{(kl)}_{i'}$; i.e., they do not depend on the subscript, then Requirement 3 of Theorem \ref{thm:exactness_indeterminate} is immediately satisfied.  This is because for all permutations $\sigma$ one has
\begin{equation} \label{eq:requirement3satisfied}
\begin{aligned}
\sum_{i,j}  u^{(ij)}_{\sigma(i)}+u^{(ij)}_{\sigma(j)} + v^{(ij)}_{\sigma^{-1}(i)}+v^{(ij)}_{\sigma^{-1}(j)} = \sum_{i,j}  u^{(ij)}_{i}+u^{(ij)}_{j} + v^{(ij)}_{i}+v^{(ij)}_{j} = \sum_{i,j} A_{ij}B_{ij},
\end{aligned}
\end{equation}
where the last equality follows from Requirement 2.

\section{Exactness Conditions for Specific Families of Matrices $A$ and $B$} \label{sec:specificab}

The intention of stating Theorem \ref{thm:exactness_indeterminate} is that it serves as a {\em master theorem} for which one can substitute reasonable choices of $\{\bu^{(ij)}\}$ and $\{\bv^{(kl)}\}$ (as functions of $A$ and $B$) and derive more specific families of matrices $A$ and $B$ for which the relaxation \eqref{eq:qap-sdp} is exact.  The presence of the values $u^{(ij)}_k$ and $v^{(kl)}_i$ make Theorem \ref{thm:exactness_indeterminate} slightly more difficult to appreciate.  The reason we state this particular version is that the added flexibility allow us to capture many instances where the relaxation is exact.  Nevertheless, in what follows, we will make specific choices of $\{\bu^{(ij)}\}$ and $\{\bv^{(kl)}\}$ that lead to families of matrices $A$ and $B$ for which the relaxation is exact.   

\subsection{Perturbation Model}

Our first application concerns a perturbation model.  Suppose $B = - A$.  It is clear that the identity matrix $I$ is an optimal solution.  Suppose instead $B \approx - A$, then one would expect $I$ to remain optimal.  This is the essence of the next result.  


\begin{theorem}\label{thm:exact-perturbation}
Suppose $A = C + \Delta$ and $B = -C + \Delta$ for symmetric matrices $C$ and $\Delta$.  Suppose one has
\begin{equation} \label{eq:perturbation_assumption}
2 (\Delta^2_{ij} + \Delta^2_{kl}) \leq \left( C_{ij} - C_{kl} + \Delta_{ij} + \Delta_{kl} \right)^2 = \left( A_{ij} - (-B_{kl}) \right)^2,
\end{equation}
for all $i,j,k,l \in [n]$ such that $i \not= j$, $k \not= l$ or $i=j$, $k=l$.  Then the tuple $(X,\tX) = (I, \vec(I) \vec(I)^T)$ is an optimal solution to \eqref{eq:qap-sdp}. 
\end{theorem}
\begin{proof}[Proof of Theorem \ref{thm:exact-perturbation}]
We make the following choices of $\bu^{(ij)}$ and $\bv^{(kl)}$ for all $i,j,k,l \in [n]$:
$$
u^{(ij)}_k = - (C_{ij}+\Delta_{ij})^2/4 + \Delta_{ij}^2/2, ~ v^{(kl)}_{i} = - (- C_{kl} + \Delta_{kl})^2/4 + \Delta_{kl}^2/2.
$$
We check that these choices satisfy the requirements stated in Theorem \ref{thm:exactness_indeterminate}.  First, one has
$$
\begin{aligned}
& u^{(ij)}_k+u^{(ij)}_l + v^{(kl)}_i+v^{(kl)}_j - A_{ij}B_{kl} \\
= ~ &  - (C_{ij}+\Delta_{ij})^2/2 + \Delta_{ij}^2 - (- C_{kl} + \Delta_{kl})^2/2 + \Delta_{kl}^2 - (C_{ij}+\Delta_{ij})(-C_{kl}+\Delta_{kl}) \\
= ~ & (\Delta_{ij}^2+\Delta_{kl}^2) - \left( C_{ij}-C_{kl}+\Delta_{ij}+\Delta_{kl} \right)^2/2.
\end{aligned}
$$
Then the assumption \eqref{eq:perturbation_assumption} implies that the above expression is non-negative for $i\not=j$, $k \not= l$ and $i=j$, $k=l$, which is Requirement 1.  Suppose that $i=k$ and $j=l$.  Then the above expression evaluates to zero, so we have Requirement 2.  Last, since the values of $u^{(ij)}_k$ and $v^{(kl)}_i$ do not depend on the subscripts, by the above remark (see \eqref{eq:requirement3satisfied}), Requirement 3 is satisfied.
\end{proof}



\subsection{Consequences for Graph Matching}

The second result concerns the graph matching problem. Let $\mathcal{G}_A$ and $\mathcal{G}_B$ be two undirected graphs over $n$ vertices.  Let $A$ and $-B$ be the adjacency matrices of the graphs $\mathcal{G}_A$ and $\mathcal{G}_B$ respectively.  Consider the graph matching problem instance formulated as a QAP in \eqref{eq:qap_graphmatching}. 

\begin{theorem}\label{thm:subgraph-isomorphism}
Let $A$ and $-B$ be the adjacency matrix of two graphs $\mathcal{G}_A$ and $\mathcal{G}_B$.  Suppose that $\mathcal{G}_A$ is a subgraph of $\mathcal{G}_B$.  Then the tuple $(X,\tX) = (I, \vec(I) \vec(I)^T)$ is an optimal solution to \eqref{eq:qap-sdp}.
\end{theorem}

\begin{proof}[Proof of Theorem \ref{thm:subgraph-isomorphism}]
We choose $\bu^{(ij)} = - A_{ij} \bone/2$ for all $i, j$, and $\bv^{(kl)} = \bzero$ for all $k, l$.  We check that this choice satisfy the requirements stated in Theorem \ref{thm:exactness_indeterminate}. Note that
\begin{equation} \label{eq:subgraph_substitution}
u^{(ij)}_k+u^{(ij)}_l + v^{(kl)}_i+v^{(kl)}_j = u^{(ij)}_k+u^{(ij)}_l = - A_{ij} .
\end{equation}

First, let $i \neq j$ and $k \neq l$.  Note that $B_{kl} \ge -1$.  By checking cases, one has $ - A_{ij} \le A_{ij}B_{kl}$.  Suppose instead that $i=j$ and $k=l$.  Then one has $-A_{ii} = 0 = A_{ii}B_{kk}$.  By combining all of these cases with \eqref{eq:subgraph_substitution}, we see that Requirement 1 is satisfied.

Second, for $i, j \in [n]$, consider the following cases.  Suppose $A_{ij} = 0$.  Then one has $-A_{ij} = A_{ij}B_{ij}$.  Suppose $A_{ij} = 1$.  Since $\mathcal{G}_A$ is a subgraph of $\mathcal{G}_B$, it must be that $B_{ij} = -1$, and hence $-A_{ij} = A_{ij}B_{ij}$.  Consequently, from \eqref{eq:subgraph_substitution}, one has $u^{(ij)}_i+u^{(ij)}_j + v^{(ij)}_i+v^{(ij)}_j = u^{(ij)}_i+u^{(ij)}_j = - A_{ij} = A_{ij}B_{ij}$, and hence Requirement 2 is satisfied. 

Last, the values of $u^{(ij)}_k$ and $v^{(kl)}_i$ do not depend on the subscripts, and hence by \eqref{eq:requirement3satisfied}, Requirement 3 is satisfied.
\end{proof}

\subsection{A Type of Comonotonicity in Matrices}

The third result concerns a type of comonotonicity in matrices.  Let $A$ and $B$ be symmetric matrices.  Suppose $A$ and $B$ satisfy the following inequality
\begin{equation} \label{eq:krushevski}
(A_{ij}-A_{kl})(B_{ij}-B_{kl}) \leq 0 \qquad \text{ for all } \qquad i,j,k,l \in [n].
\end{equation}
A consequence of the rearrangement inequality is that the identity matrix $I$ is an optimal solution to the QAP -- this observation was made early on by Krushevski (Theorem 4.14 in \cite{qapbookcela}).  

Our next result shows that the semidefinite relaxation \eqref{eq:qap-sdp} is exact under a similarly looking but stronger condition:

\begin{theorem}\label{thm:amiability}
Suppose $A$ and $B$ are matrices satisfying
\begin{equation} \label{eq:amiability}
A_{ij}B_{ij} + A_{kl}B_{kl} \leq 2 \min\{A_{ij}B_{kl},A_{kl}B_{ij}\},
\end{equation}
for all $i,j,k,l \in [n]$.  Then $(I, \vec(I) \vec(I)^T)$ is an optimal solution to \eqref{eq:qap-sdp}. 
\end{theorem}

Note that the condition \eqref{eq:krushevski} is equivalent to requiring that $ A_{ij} B_{ij} + A_{kl} B_{kl} \leq A_{ij} B_{kl} + A_{kl} B_{ij}$.  Therefore \eqref{eq:amiability} simply requires that the LHS of \eqref{eq:krushevski} is less than (twice) the minimum of the two terms on the RHS.  We are not aware if QAP instances whereby the matrices $A$ and $B$ satisfy the weaker condition \eqref{eq:krushevski} can be solved tractably.  

\begin{proof}[Proof of Theorem \ref{thm:amiability}]
Set $u^{(ij)}_k = A_{ij}B_{ij}/4$ and $v^{(kl)}_{i} = A_{kl}B_{kl}/4$ for all $i,j,k \in [n]$.  We show that $\bu^{(ij)}$ and $\bv^{(kl)}$ constructed in this way satisfy the requirements stated in Theorem \ref{thm:exactness_indeterminate}.

One has $ u^{(ij)}_k+u^{(ij)}_l + v^{(kl)}_i+v^{(kl)}_j = (A_{ij}B_{ij} + A_{kl}B_{kl})/2 \le \min\{A_{ij}B_{kl},A_{kl}B_{ij}\} \leq A_{ij}B_{kl}$.  This proves Requirement 1 and 2.  Requirement 3 is satisfied because the values of $u^{(ij)}_k$ and $v^{(kl)}_i$ do not depend on the subscripts (see \eqref{eq:requirement3satisfied}). 
\end{proof}

\section{Proof of Main Results} \label{sec:proof_exactness}

\subsection{Duality}

The process of establishing conditions that guarantee exactness of any convex relaxation typically starts with the dual.  The difficulty with analyzing \eqref{eq:qap-sdp} directly is that it incorporates many different types of constraints, which typically results in a dual formulation that can be very intimidating to work with.  We take a slightly more forgiving approach.  We instead consider the sub-problem of \eqref{eq:qap-sdp} in which $X \in \mathrm{DS}(n)$ is {\em fixed}, and we only optimize over $\tX$.  To improve readability, we denote $\bx := \vec(X)$.  Also, we let $I:=I_n$ denote the $n\times n$ identity matrix, we let $E:=E_n \in \R^{n\times n}$ be the matrix whose entries are all ones, and we let $\bone := \bone_n \in \R^n$ be the vector whose entries are all ones.  (In settings where the dimensions are not $n\times n$ or $n$ respectively, we will specify the dimensions.) We let $E_{ij}$ represent the standard basis matrix whose $(i,j)$-th entry is one and whose remaining entries are zero. 

We rewrite the affine constraints in the second row of \eqref{eq:qap-sdp} as follows:
$$
\begin{aligned}
\langle E\otimes E_{kl},\tX \rangle = 1 \text{ and } \langle E_{kl} \otimes E, \tX \rangle = 1 \text{ for all $k,l \in [n]$}.
\end{aligned}
$$
Consider the following equivalent formulation of the sub-problem of \eqref{eq:qap-sdp}: 
\begin{equation} \label{eq:qap-sdp_Xsub}
\begin{aligned}
\min_{\tX } \qquad & \langle A \otimes B, \tX \rangle \\
\mathrm{s.t.} \qquad & \tX \succeq \bx \bx^T \\
& \langle E\otimes E_{kl},\tX \rangle = 1 \text{ and } \langle E_{kl} \otimes E, \tX \rangle = 1 \text{ for all $k,l \in [n]$}\\
& \tX_{(i,k),(j,k)} = 0, \tX_{(k,i),(k,j)} = 0  \text{ for all } k \text{ and } i \neq j \\
& \tX \ge 0
\end{aligned}.
\end{equation}
Notice that \eqref{eq:qap-sdp_Xsub} still specifies a SDP.  We state the dual program in the following.

\begin{proposition}\label{prop:qap-inner}
The dual program to \eqref{eq:qap-sdp_Xsub} is given by
\begin{equation}\label{eq1:prop:qap-inner} 
\begin{aligned}
\underset{\tZ,G,H}{\max} \qquad & \bx^T \big(A \otimes B + G \otimes E + E \otimes H - \tZ \big) \bx - \langle G+H,E \rangle \\
\mathrm{s.t.} \qquad & A \otimes B + G \otimes E + E \otimes H \succeq \tZ \\
& \tZ_{(i,k),(j,l)} \ge 0 \text{ whenever } (i \neq j, k \neq l) \text{ or } (i=j,k=l)
\end{aligned}.
\end{equation}
Moreover, strong duality holds for all inputs $\bx$. 
\end{proposition}

Here, $G,H \in \mathbb{R}^{n \times n}$, while $\tZ \in \mathbb{R}^{n^2 \times n^2}$.  As a small note: In the primal \eqref{eq:qap-sdp_Xsub} we have an equality condition that applies to all indices $(i,k),(j,l)$ whenever $i=j$ and $k \neq l$, or when $i\neq j$ and $k = l$.  The dual imposes an inequality on the complement of these indices, and one can check that the complement is indeed those that satisfy $i \neq j$ and $k \neq l$, or $i=j$ and $k=l$.


To simplify the notation, we denote the dual objective function by $f$
$$
f(\bx,G,H,\tZ) := \bx^T \big(A \otimes B + G \otimes E + E \otimes H - \tZ \big) \bx - \langle G+H,E \rangle.
$$

\begin{proof}[Proof of Proposition \ref{prop:qap-inner}] The proof comprises two parts.

[Derive dual formulation]:  The first step is to derive the dual.  Define the dual variables as follows
    \begin{equation*}
        \begin{aligned}
            & \tY \succeq 0 && : \tX \succeq \bx \bx^T \\
            & H \in \R^{n \times n}
            && : \langle E\otimes E_{kl},\tX \rangle = 1 \text{ for all $k,l \in [n]$}\\
            & G \in \R^{n \times n}
            && : \langle E_{kl} \otimes E, \tX \rangle = 1 \text{ for all $k,l \in [n]$}\\
            & \tZ_{(i,k),(j,l)} \ge 0 && : \tX_{(i,k),(j,l)} \ge 0 \text{ if } (i \neq j, k \neq l) \text{ or } (i=j,k=l) \\
            & \tZ_{(i,k),(j,l)} \in \mathbb{R} && : \tX_{(i,k),(j,l)} = 0 \text{ if } (i=j, k \neq l) \text{ or } (i \neq j,k=l) \\
        \end{aligned} .
    \end{equation*}
    The dual function is given by
    \begin{equation*}
        \begin{aligned}
            \theta(\tY,G,H,\tZ) & := \min_{\tX} \quad \tr (A \otimes B \cdot \tX) - \tr(\tY(\tX - \bx \bx^T)) - \tr(\tZ\tX) \\
            & \quad \quad + \sum_{i,j} H_{ij} \big( \langle E\otimes E_{ij},\tX \rangle - 1 \big) + \sum_{i,j} G_{ij} \big(\langle E_{ij} \otimes E, \tX \rangle - 1 \big) \\
            & = \min_{\tX} \quad \tr \big( (A \otimes B + G \otimes E + E \otimes H - \tZ - \tY)\tX \big) + \bx^T \tY \bx - \langle G+H,E \rangle. 
        \end{aligned} .
    \end{equation*}
This optimization instance is unconstrained over $\tX$.  As such it is unbounded below if
$$A \otimes B + G \otimes E + E \otimes H - \tZ - \tY \neq 0.$$ 
Subsequently, the dual function simplifies to
        \begin{equation*}
            \theta(\tY,G,H,\tZ) = \begin{cases}
                \bx^T \tY \bx - \langle G+H,E \rangle & \text{if $\tY = A \otimes B + G \otimes E + E \otimes H  - \tZ$} \\
                -\infty &  \text{otherwise}
            \end{cases}.
        \end{equation*}
    By incorporating the constraints on $\tY$ and $\tZ$, we conclude that the dual program of \eqref{eq:qap-sdp_Xsub} is \eqref{eq1:prop:qap-inner}.
    
[Strong duality]:  The next step is to establish strong duality.  We do so by applying Slater's condition.  
Specifically, we show that the relative interior of the feasible region is non-empty; i.e., there exists feasible solutions that satisfy the PSD constraint and the inequality strictly.  

Define $t = \max \{0, -\lambda_{\rm min}(A \otimes B)\}$.  One has 
    \begin{equation}\label{eq2:prop:qap-inner}
    \tZ := (t+2)E_{n^2} - (t+1)I_{n^2} \ge E_{n^2} > 0.
    \end{equation}
    Also, define 
    $$
    G = \bzero_{n \times n}, ~ H = (t+2) E.
    $$
Note that $(E \otimes \bone)(I \otimes \bone^T) = E \otimes \bone \bone^T = E_{n^2}$.  Subsequently, we have
$$
    \begin{aligned}
     A \otimes B + G \otimes E + E \otimes H - \tZ = & ~ A \otimes B + (t+2) E_{n^2} - \tZ \\
    = & ~ A \otimes B + (t+1) I_{n^2} 
    \end{aligned},
$$
where the last line follows from \eqref{eq2:prop:qap-inner}.  Since $t \geq -\lambda_{\rm min}(A \otimes B)$, we have $A \otimes B + (t+1) I_{n^2} \succeq  I_{n^2} \succ 0$.  Thus Slater's condition is satisfied, and we conclude strong duality. 
\end{proof}

\subsection{Exactness Conditions Based on Duality}

Our next result is to apply Proposition \ref{prop:qap-inner} and state a set of conditions under which the relaxation \eqref{eq:qap-sdp} is exact.  We begin by defining the following sets:
$$
\begin{aligned}
\mathcal{S} & := \mathrm{cl}(\{ \tX : G \otimes E + E \otimes H + \tX \succeq 0 \}), \\
\mathcal{T} & :=  \{ \tX : \vec(X)^T \tX \vec(X) \geq \vec(I)^T \tX \vec(I) ~ \text{for all} ~ X \in \mathrm{DS}(n) \}.
\end{aligned}
$$

\begin{proposition}\label{thm:tight-closure}
Suppose $I \in \text{Perm}(n)$ is an optimal solution to \eqref{eq:qap}.  Suppose there exists $\tZ \in \R^{n^2 \times n^2}$ such that the following conditions are satisfied:
\begin{enumerate}
    \item[(i)] $\tZ_{(i,k),(j,l)} \geq 0$ for all ($i \neq j,k \neq l$) and all ($i=j,k=l$),
    \item[(ii)] $\tZ_{(i,i),(j,j)}=0$ for all $i,j$, 
    \item[(iii)] $A \otimes B - \tZ \in \mathcal{S}$, and 
    \item[(iv)] $A \otimes B - \tZ \in \mathcal{T}$.
\end{enumerate}
Then $I$ is also an optimal solution to \eqref{eq:qap-sdp}; moreover, the optimal values of \eqref{eq:qap} and \eqref{eq:qap-sdp} are equal. 
\end{proposition}

We point out that the set $\{ \tX : G \otimes E + E \otimes H + \tX \succeq 0 \}$ is {\em not} closed.  Moreover, there are examples of matrices $A$ and $B$ that necessarily require choices of $\tZ$ for which $A \otimes B - \tZ$ lines on the boundary -- we discuss this point in Section \ref{sec:geom}.

We proceed to prove Proposition \ref{thm:tight-closure}.  We need the following observation.  Suppose $X \in \text{DS}(n)$, and we let $\bx = \vec(X)$.  Then $(\bone^T \otimes I) \bx = \bone$ and $(I \otimes \bone^T) \bx = \bone$ where $\bx = \vec(X)$. This gives
$$
\bx^T (G \otimes E) \bx = \bx^T (G \otimes \bone \cdot I \otimes \bone^T) \bx = \bx^T G \otimes \bone \cdot \bone = \bx^T \vec(\bone \bone^T G) = \langle X,\bone \bone^T G \rangle = \langle G,E \rangle.
$$
Similarly, one has $\bx^T (E \otimes H) \bx = \langle H,E \rangle$).  Then it follows that
\begin{equation} \label{eq:piab_vanishing}
f(\bx,G,H,\tZ) = \bx^T (A \otimes B-\tZ) \bx  \quad \text{ whenever } \quad \bx = \vec(X) \text{ for some } X \in \text{DS}(n).
\end{equation}

\begin{proof}[Proof of Proposition \ref{thm:tight-closure}]
First, we define the function $g: \R^{n \times n} \rightarrow \R$
\begin{equation} \label{eq:g_definition}
\begin{aligned}
g(X) ~ := ~ \max_{G,H,\tZ} ~ f(\vec(X),G,H,\tZ) \quad \mathrm{s.t.} \quad &  A \otimes B + G \otimes E + E \otimes H \succeq \tZ \\ 
& \tZ_{(i,k),(j,l)} \ge 0 \text{ for all $i\not=j,k\not=l$ and $i=j,k=l$}
\end{aligned}.
\end{equation}

Second, it is given that $A \otimes B - \tZ \in \mathcal{S}$ (condition (iii)).  This means that there exists a sequence $\{(\tZ_k, G_k, H_k) \}_{k=1}^{\infty}$ such that (1) $(A \otimes B - \tZ_k) + (G_k \otimes E + E \otimes H_k) \succeq 0$ for all $k$, (2) $\tZ_k \rightarrow \tZ$.  It follows from (1) and condition (i) that $\{G_k,H_k,\tZ_k\}$ are feasible solutions to \eqref{eq:g_definition} for all $k$, so we have 
$$
g(X) \geq f(\vec(X), G_k, H_k, \tZ_k) \quad \text{ for all } \quad X \in \mathrm{DS}(n).
$$
One then has
\begin{equation*}
\begin{aligned}
\text{The opt. val. of \eqref{eq:qap-sdp}} & \overset{(a)}{=}  \min_{X \in \text{DS}(n)} g(X) \\
& \geq \min_{X \in \text{DS}(n)} f(\vec(X), G_k, H_k, \tZ_k) \\
& \overset{(b)}{=} \min_{X \in \text{DS}(n)} \vec (X)^T (A \otimes B - \tZ_k) ~ \vec (X),
\end{aligned}
\end{equation*}
for all $k \in \mathbb{N}$.  Here, (a) is from Proposition \eqref{prop:qap-inner}, while (b) is from \eqref{eq:piab_vanishing}.  By taking the limit $k \rightarrow \infty$, one has
$$
\text{The opt. val. of \eqref{eq:qap-sdp}} \geq \lim_{k \rightarrow \infty} \Big( \min_{X \in \mathrm{DS}(n)}  ~ \vec (X)^T (A \otimes B - \tZ_k)  \vec (X) \Big).
$$
We can view $\vec (X)^T (A \otimes B - \tZ_k)  \vec (X)$ as a function where $X$ is the argument.  Then $\{\vec (X)^T (A \otimes B - \tZ_k)  \vec (X)\}_{k}$ defines a sequence of continuous functions over $\mathrm{DS}(n)$, which is compact.  Since $\tZ_k \rightarrow \tZ$, it follows that this sequence converges uniformly to the function $\vec (X)^T (A \otimes B - \tZ)  \vec (X)$.  This allows us to conclude that
$$
\lim_{k \rightarrow \infty} \min_{X \in \mathrm{DS}(n)}  ~ \vec (X)^T (A \otimes B - \tZ_k)  \vec (X) ~=~  \min_{X \in \mathrm{DS}(n)} ~ \vec (X)^T (A \otimes B - \tZ)  \vec (X).
$$

Finally, we note that
\begin{equation*}
\begin{aligned}
\min_{X \in \mathrm{DS}(n)} ~ \vec (X)^T (A \otimes B - \tZ)  \vec (X)  & ~\overset{(c)}{=}~ \vec (I)^T (A \otimes B - \tZ) \, \vec (I) \\
& ~\overset{(d)}{=}~ \vec (I)^T A \otimes B \, \vec (I) \\
& ~\geq~ \text{The opt. val. of \eqref{eq:qap}}.
\end{aligned}
\end{equation*}
Here, (c) follows from condition (iv), and (d) follows from condition (ii).  Now note that the optimal value of \eqref{eq:qap-sdp} is a lower bound to the optimal value of \eqref{eq:qap}, and hence all of these inequalities are in fact equalities, from which we conclude our result.
\end{proof}

\subsection{Further Simplifications of $\mathcal{S}$ and $\mathcal{T}$} \label{sec:s_and_t}

Our next step is to make clever choices of the operator $\tZ$ for which we can certify exactness for certain families of matrices $A$ and $B$.  Unfortunately, the conditions in Proposition \ref{thm:tight-closure}, as it is stated, are not easy to work with.  Our next few steps are to find simpler characterizations of sets that belong to $\mathcal{S}$ and $\mathcal{T}$.

We start with $\mathcal{S}$.  It is straightforward to see that this is a closed, convex cone.  

\begin{lemma}\label{thm:t_cone_description1}
Let $M = \bu \bone^T + \bone \bu^T$, and let $\tP = (E_{ij}+E_{ji}) \otimes M \in \bbR^{n^2\times n^2}$ for some $i,j \in [n]$.  Then $\tP \in \mathcal{S}$.
\end{lemma}

\begin{proof}[Proof of Lemma \ref{thm:t_cone_description1}]

We split the proof cases $i = j$ and $i \neq j$. 

[Case $i \neq j$]:  Without loss of generality we may take $i=1$ and $j=2$.  We show that there exists a sequence of tuples of matrices $\{G_k,H_k,\tP_k\}_{k=1}^\infty$ satisfying (i) $G_{k} \otimes E + E \otimes H_k + \tP_k \succeq 0 $, and (ii) $\tP_k \rightarrow \tP$ as $k \rightarrow \infty$. In what follows, we denote $\bzero = (0,\ldots,0)^T \in \mathbb{R}^{n}$ (we will specify the dimensions if it is not $n$). And for $k > 0$, we define
$$
\begin{aligned}
\bv_{1,k} ~:=~ & (\bone - \bu/k, - \bone + \bu/k, \bzero ,\ldots, \bzero)^T, \\
\bv_{2,k} ~:=~ & (\bzero, \bzero, \bone - \bu/k, \ldots, \bone - \bu/k)^T, \\
\bv_{3,k} ~:=~ & (\bone + \bu/k, \bone + \bu/k, - \bone - \bu/k, \ldots, \ldots, - \bone - \bu/k)^T. 
\end{aligned}
$$
Define $\tV_k := (k/4) (2 \bv_{1,k} \bv_{1,k}^T + 2 \bv_{2,k} \bv_{2,k}^T + \bv_{3,k} \bv_{3,k}^T)$.  Then one has
$$
\tV_k = 
\frac{k}{4}\left(
\begin{array}{c|c}
    \begin{array}{cc}
        3E - M/k & - E + 3M/k \\
        - E + 3M/k & 3E - M/k
    \end{array} & E_{2 \times (n-2)} \otimes (-E - M/k) \\
    \hline
    E_{2 \times (n-2)} \otimes (-E - M/k) & E_{(n-2) \times (n-2)} \otimes (3E - M/k )
\end{array}
\right) + O(1) E_{n^2 \times n^2}.
$$
Similarly, define 
$$
G_k := \frac{k}{4}\left(
\begin{array}{c|c}
    \begin{array}{cc}
        3 & -1 \\
        -1 & 3
    \end{array} & E_{2 \times (n-2)} \\
    \hline
    E_{2 \times (n-2)} & E_{(n-2) \times (n-2)}
\end{array}
\right), \qquad H_k := -\frac{M}{4},
$$ 
and let $\tP_k := \tV_k - G_k \otimes E - E \otimes H_k$.  Then note that
$$
\tP_k = \tP + O(1) E_{n^2 \times n^2},
$$
and in particular, $\tP_k \rightarrow \tP$ as $k \rightarrow \infty$.  Moreover, since $\tV_k$ is the sum of three rank-one PSD components, one has $\tV_k \succeq 0$, from which one has $G_k \otimes E + E \otimes H_k + \tP_k \succeq 0$.  This concludes the result. 

[Case $i=j$]:  The proof in this case is similar.  Without loss of generality we take $i=j=1$.  Define
\begin{equation*}
\begin{aligned}
\bv_k ~:=~ & (\bone+\bu/k,\bzero,\ldots,\bzero)^T, \\
\tV_k ~:=~ & k \bv_k \bv_k^T = k \left(
\begin{array}{c|c}
    E + M/k & \bzero_{n \times n(n-1)} \\
\hline
    \bzero_{n(n-1) \times n} & \bzero_{n(n-1)\times n(n-1)}
\end{array}
\right) + O(1)E_{n^2\times n^2}, \\
G_k ~:=~ & k \left(
\begin{array}{c|c}
    1 & \bzero_{1 \times (n-1)} \\
\hline
    \bzero_{(n-1) \times 1} & \bzero_{(n-1)\times (n-1)}
\end{array}
\right).
\end{aligned}
\end{equation*}
This time, $\tV_k$ is rank-one PSD.  Let $\tP_k := \tV_k - G_k \otimes E$, from which we conclude the result.
\end{proof}

%

\begin{proposition}\label{thm:inside_s}
Suppose $\tP$ is of the form
$$
\tP = \sum_{i, j} (E_{ij}+E_{ji}) \otimes ( \bu^{(ij)} \bone^T + \bone \bu^{(ij)T})  
+ ( \bv^{(ij)} \bone^T + \bone \bv^{(ij)T}) \otimes (E_{ij}+E_{ji}),
$$
Then $\tP \in \mathcal{S}$.
\end{proposition} 
\begin{proof}
Recall that $\mathcal{S}$ is a cone.  As such, it suffices to show that each of the summands $(E_{ij}+E_{ji}) \otimes ( \bu^{(ij)} \bone^T + \bone \bu^{(ij)T})$ belong to $\mathcal{S}$.  But this is precisely Lemma \ref{thm:t_cone_description1}.
\end{proof}

\begin{lemma}\label{thm:simplify_T}
Suppose
\begin{equation}\label{eq:z_form}
\tP = \sum_{i, j} (E_{ij}+E_{ji}) \otimes ( \bu^{(ij)} \bone^T + \bone \bu^{(ij)T})  
+ ( \bv^{(ij)} \bone^T + \bone \bv^{(ij)T}) \otimes (E_{ij}+E_{ji}).
\end{equation}
Then $\tP \in \mathcal{T}$ if and only if the following inequality holds for all permutations $\sigma$
\begin{equation}\label{eq:iv_simplified}
\sum_{i,j}  u^{(ij)}_{\sigma(i)}+u^{(ij)}_{\sigma(j)} + v^{(ij)}_{\sigma^{-1}(i)}+v^{(ij)}_{\sigma^{-1}(j)} \geq \sum_{i, j} \bu^{(ij)}_{i} + \bu^{(ij)}_{j} + \bv^{(ij)}_{i}+\bv^{(ij)}_{j}.
\end{equation}
\end{lemma}
\begin{proof}[Proof of Lemma \ref{thm:simplify_T}]
To simplify notation, let's suppose $\tP = (E_{ij}+E_{ji}) \otimes ( \bu \bone^T + \bone \bu^{T}) $.  Let $X \in \mathrm{DS}(n)$.  One then has
\begin{equation*}
\begin{aligned}
& \vec(X)^T (E_{ij}+E_{ji}) \otimes ( \bu \bone^T + \bone \bu^T)  \vec(X) \\
=~ & \mathrm{tr}( X^T ( \bu \bone^T + \bone \bu^T)  X (E_{ij}+E_{ji}) ) \\
=~ & \mathrm{tr}( X^T ( \bu \bone^T )  X (E_{ij}+E_{ji}) ) + \mathrm{tr}( X^T ( \bone \bu)  X (E_{ij}+E_{ji}) ) \\
=~ & \mathrm{tr}( X^T ( \bu \bone^T )  (E_{ij}+E_{ji}) ) + \mathrm{tr}(  \bone \bu^T  X (E_{ij}+E_{ji}) ).
\end{aligned}
\end{equation*}
In the last line, we used the fact that $X^T \bone = \bone$ since $X$ is doubly stochastic.  By noting cyclicity of trace, the last line simplifies to
\begin{equation*}
(\be_i + \be_j)^T X^T  \bu + \bu^T  X (\be_i + \be_j) = 2 \langle X, \bu (\be_i+\be_j)^T \rangle.
\end{equation*}
More generally, suppose $\tP$ has the form \eqref{eq:z_form}.  Then $\tP \in \mathcal{T}$ if and only if
\begin{equation*}
\langle X - I, \sum_{i,j} \bu^{(ij)} (\be_i+\be_{j})^T + (\be_{i}+\be_{j}) \bv^{(ij)T} \rangle \geq 0,
\end{equation*}
for all $X \in \mathrm{DS}(n)$.  But because the above is a linear constraint, by the Birkhoff-von Neumann's theorem, it suffices to check that this holds for all permutation matrices $X$.  This then implies \eqref{eq:iv_simplified} for all permutations $\sigma$.

In the reverse direction, we simply check that every implication holds in both directions.  We omit this part of the proof.
\end{proof}

\subsection{Proof of Theorem \ref{thm:exactness_indeterminate}}

\begin{proof}[Proof of Theorem \ref{thm:exactness_indeterminate}]
We complete the proof using Proposition \ref{thm:tight-closure}. Define
\begin{equation} \label{eq:construction_of_P}
\tP = \sum_{i, j} (E_{ij}+E_{ji}) \otimes ( \bu^{(ij)} \bone^T + \bone \bu^{(ij)T})  
+ ( \bv^{(ij)} \bone^T + \bone \bv^{(ij)T}) \otimes (E_{ij}+E_{ji}).
\end{equation}
Define $\tZ = A \otimes B - \tP$.  Then
$$
\tZ_{(i,k),(j,l)} = A_{ij} B_{kl} - \tP_{(i,k),(j,l)} = A_{ij} B_{kl} - ( u^{(ij)}_{k} + u^{(ij)}_{l} + v^{(kl)}_{i} + v^{(kl)}_{j} ).
$$
Requirement 1 is satisfied by Assumption 1.  Requirement 2 is satisfied by Assumption 2.  Requirement 3 follows from Proposition \ref{thm:inside_s} on $\tP$.  Requirement 4 follows from Lemma \ref{thm:simplify_T} and Assumption 3.  This proves the result.
\end{proof}

\section{Geometric Aspects of the Dual Problem} \label{sec:geom}

One technical difficulty we encountered in our analysis is that the set of dual feasible solutions is {\em not} closed.  This meant that there are specific families of matrices for the semidefinite relaxation \eqref{eq:qap-sdp} is exact, but there does not exist any dual feasible solutions that certify exactness.  In Section \ref{sec:s_and_t}, we address this technical difficulty by constructing a sequence of dual feasible solutions such that the dual objective converges to the primal optimal value.  We discuss this aspect in greater detail.

Consider the following pair of matrices 
$$
A = \left( \begin{array}{cccc}
0 & 1 & 0 & 0 \\ 1 & 0 & 0 & 0 \\ 0 & 0 & 0 & 1 \\ 0 & 0 & 1 & 0 
\end{array}\right),
B = - \left( \begin{array}{cccc}
0 & 1 & 0 & 0 \\ 1 & 0 & 1 & 0 \\ 0 & 1 & 0 & 0 \\ 0 & 0 & 0 & 0 
\end{array}\right).
$$

Consider the following variable $\tP$
\begin{equation*}
\tP = 
\left( \begin{array}{cccc}
0 & P & 0 & 0 \\ P & 0 & 0 & 0 \\ 0 & 0 & 0 & P \\ 0 & 0 & P & 0 
\end{array} \right)
\qquad \text{where} \qquad
P = - \left( \begin{array}{cccc}
0 & 1 & 0 & 0 \\ 1 & 2 & 1 & 1 \\ 0 & 1 & 0 & 0 \\ 0 & 1 & 0 & 0 
\end{array} \right).
\end{equation*}
We define $\tZ$ such that $A \otimes B - \tZ = \tP$.  One can check that this specific choice of $\tZ$ satisfies the conditions in Proposition \ref{thm:tight-closure}: (i) $\tZ_{(i,j),(k,l)} \geq 0$ for all ($i \neq k,j \neq l$) and all ($i=k,j=l$), (ii) $\tZ_{(i,i),(j,j)}=0$ for all $i,j$, and (iv) $\tP = A \otimes B - \tZ \in \mathcal{T}$.  (Indeed, it is clear that these conditions specify closed convex sets since they are specified as linear inequalities.)  One has the following:

\begin{proposition}
One has $\tP \in \mathrm{cl}(\{ \tX : G \otimes E + E \otimes H + \tX \succeq 0 \})$.
\end{proposition}

In what follows, $\be_i$ denotes the standard basis vector whose $i$-th coordinate is one and whose remaining entries are zero.

\begin{proof}
This is a direct application of Proposition \ref{thm:inside_s}, and in fact, Lemma \ref{thm:t_cone_description1}.  In particular, one can write $P = - \be_2 \bone^T - \bone \be_2^T$, which is the form in Lemma \ref{thm:t_cone_description1}, from which the result follows.
\end{proof}

On the other hand:

\begin{proposition}
One has $\tP \notin \{ \tX : G \otimes E + E \otimes H + \tX \succeq 0 \}$.
\end{proposition}

\begin{proof}
We need to show that there does not exist any choices of $G$ and $H$ such that $ G \otimes E + E \otimes H + \tP$ is PSD.  We apply a change of coordinates via the basis 
\begin{equation*}
\tU = 
\frac{1}{\sqrt{2}}\left( \begin{array}{cccc}
I & I & 0 & 0 \\ I & -I & 0 & 0 \\ 0 & 0 & I & I \\ 0 & 0 & I & -I 
\end{array} \right), \qquad \text{ where } \qquad \tU \tP \tU^T = 
\left( \begin{array}{cccc}
P & 0 & 0 & 0 \\ 0 & -P & 0 & 0 \\ 0 & 0 & P & 0 \\ 0 & 0 & 0 & -P
\end{array} \right).
\end{equation*}

Under this change of basis, matrices of the form $G \otimes E$ and $E\otimes H$ are transformed as follows.  First, $\tU (G \otimes E) \tU^T = \tilde{G} \otimes E$ for some $\tilde{G}$; that is, it retains the same structure.  Second, 
\begin{equation*}
\tU (E \otimes H) \tU^T = \left( \begin{array}{cccc}
2H & 0 & 2H & 0 \\ 0 & 0 & 0 & 0 \\ 2H & 0 & 2H & 0 \\ 0 & 0 & 0 & 0
\end{array} \right)
\end{equation*}

In the new basis, we claim that the (2,2) (as well as the (4,4)) block minor (of size $4 \times 4$) is never positive semidefinite, for any choice of $\tilde{G}$ or $H$.  First, $H$ has zero contribution to the $(2,2)$ and $(4,4)$ block.  Second, $\tilde{G}$ contributes a constant in this block.  As such, our task is to show that there does not exist a scalar $s$ such that $-P+sE$ is PSD.  We show in particular that $-P+sE$ has a strictly negative eigenvalue.  

Write $-P + s E = \be_2 \bone^T + \bone \be_2^T + s \bone \bone^T$.  Consider vector $\bv = \alpha \bone + \beta \be_2$.  Then one has $(-P + s E) \bv= (\be_2 \bone^T + \bone \be_2^T + s \bone \bone^T) (\alpha \bone + \beta \be_2) = (4\alpha s + \beta s + \alpha + \beta)\bone + (4\alpha+\beta)\be_2$.  Suppose we set
$$
\alpha = s+1, \beta = -2s - 2 \sqrt{s^2+s+1}.
$$
One can verify that 
$$
\frac{\alpha}{\beta} = \frac{4\alpha s + \beta s + \alpha + \beta}{4\alpha+\beta};
$$
that is, $\bv$ is an eigenvector of $-P + s E$ with eigenvalue
$$
\frac{4\alpha s + \beta s + \alpha + \beta}{\alpha} = 2s + 1 - 2 \sqrt{s^2+s+1}.
$$
By noting that $4s^2 + 4s + 1 < 4 (s^2+s+1)$, we conclude that this eigenvalue must be negative for all $s$.  Therefore the $(2,2)$-block always contains a negative eigenvalue, which completes the proof.
\end{proof}

We point out that there are a number of prior works \cite{aflalo2015convex,ling2024exactnessqapsdp} that state exactness under the condition that the eigenvalues of the matrices are all distinct, and where the eigenvectors are not orthogonal to the all ones vector $\bone$.  It turns out that these matrices $A$ and $B$ do {\em not} satisfy these conditions.  In particular, the matrix $A$ has eigenvalues $\pm 1$, each with multiplicity two.  Moreover, the eigenspace of the matrix $A$ corresponding to the eigenvalue $-1$ are vectors of the form $(t_1,-t_1,t_2,-t_2)$, which is orthogonal to $\bone$.

Our discussion here suggests that there are technical difficulties when matrices have repeated eigenvalues or when the eigenvectors are orthogonal to $\bone$.  More specifically, our discussion suggests that the difficulties seem to arise because the dual feasible solution does not exist for such matrices.  In particular, the prior works in \cite{aflalo2015convex,ling2024exactnessqapsdp} rely on an explicit construction of a dual feasible solution to certify exactness, and as we see in this Section, such a strategy may fail in the presence of repeated eigenvalues or if $\bone$ is an eigenvector.  While generic matrices do not satisfy such conditions (say with probability zero under a suitable random model), it is quite possible for adjacency matrices of graphs to satisfy these conditions.

\section{Numerical Experiments}\label{sec:numerics}


A core theme in this paper was the strength of the semidefinite relaxation \eqref{eq:qap-sdp}, as well the ability of Theorem \ref{thm:exactness_indeterminate} to certify exactness.  We investigate how these statements hold up through numerical experiments.  In the first part, $A$ and $-B$ are adjacency matrices of graphs, and in the second part, $A$ and $-B$ are distances matrices of point clouds.  In each set of experiments, we track the following metrics: (i) the percentage of instances for which the semidefinite relaxation \eqref{eq:qap-sdp} is exact (among all generated instances), and (ii) the percentage of instances for a certificate from Theorem \ref{thm:exactness_indeterminate} exists (among all instances for which \eqref{eq:qap-sdp} is exact, as in part (i)).  All of our experiments are performed with \texttt{CVX} \cite{cvx2,cvx1} using the MOSEK solver \cite{mosek}.

\subsection{Finding a Certificate}\label{sec:find_certificate}

First we describe our procedure for finding set of vectors $\{\bu^{(ij)}\}_{i, j \in [n]}$ and $\{\bv^{(kl)}\}_{k, l \in [n]}$ satisfying the requirements of Theorem \ref{thm:exactness_indeterminate}.  Note that Requirement 3 of Theorem \ref{thm:exactness_indeterminate} specifies an inequality that holds for all possible permutations -- in principle, one has $n!$ constraints to satisfy.  In the following, we describe a simpler condition that guarantees the existence of the vectors $\{\bu^{(ij)}\}_{i, j \in [n]}$ and $\{\bv^{(kl)}\}_{k, l \in [n]}$.

\begin{proposition}\label{lem:exactness_3_equivalence}
Let $\{\bu^{(ij)}\}_{i, j \in [n]}$ and $\{\bv^{(kl)}\}_{k, l \in [n]}$ be a collection of vectors satisfying Requirements 1 and 2 of Theorem \ref{thm:exactness_indeterminate}.  Suppose in addition that
\begin{equation} \label{eq:new_requirement3}
- \Big( \sum_{i,j} \bu^{(ij)} (\be_i+\be_{j})^T + (\be_{i}+\be_{j}) \bv^{(ij)T} \Big) = -W + \bone \mathbf{p}^T + \mathbf{q} \bone^T,    
\end{equation}
for some $\mathbf{p}, \mathbf{q} \in \bbR^n$, and some non-negative matrix $W \in \bbR^{n \times n}$ with zero diagonal.  Then \eqref{eq:qap-sdp} is exact.
\end{proposition}
In our numerical experiments, we search for vectors $\bu$ and $\bv$ satisfying the first two Requirements of Theorem \ref{thm:exactness_indeterminate} in addition to \eqref{eq:new_requirement3}.


The proof of Lemma \ref{lem:exactness_3_equivalence} can be broken down into the following steps.  First, we recall the definition of the normal cone of a closed convex set $\mathcal{C}$ at $x \in \mathcal{C}$:
$$
\cN_{\mathcal{C}}(x) = \{z:\langle z, y-x \rangle \le 0 \text{ for all }y \in \mathcal{C} \}.
$$

In our context, the set $\mathcal{C}$ of interest is the set of all doubly stochastic matrices $\mathrm{DS}(n)$.

\begin{lemma}\label{lem:normalcone-description}
Suppose $M = -W + \bone \mathbf{p}^T + \mathbf{q} \bone^T \in \mathbb{R}^{n \times n}$, where $W$ is a non-negative matrix with zero diagonal, and  $\mathbf{p}, \mathbf{q} \in \bbR^n$.  Then $M \in \cN_{\mathrm{DS}(n)}(I)$. 
\end{lemma}


\begin{proof}
Recall from the Birkhoff-von Neumann theorem that the extreme points of the $\mathrm{DS}(n)$ are the permutation matrices.  Hence to show that $M \in \cN_{\mathrm{DS}(n)}(I)$, it suffices to show that $\langle M, \Pi - I \rangle = \mathrm{tr}(M^T(\Pi - I)) \leq 0$ for all permutation matrices $\Pi$.

Suppose $M = \bone \mathbf{p}^T$.  Then $\mathrm{tr}(M^T(\Pi-I)) = \mathrm{tr}(\mathbf{p}\bone^T(\Pi - I)) = \mathrm{tr}(\mathbf{p}\bone^T - \mathbf{p}\bone^T) = 0$.  The same is true if $M = \mathbf{q} \bone^T$.  

Finally suppose $M = -W$.  Then $\mathrm{tr}(MI) = -\mathrm{tr}(W) = 0$ since $W$ has zero diagonals.  However, $\mathrm{tr}(M\Pi) = -\mathrm{tr}(W\Pi) \leq 0$ because the entries of $W$ are non-negative.  Hence $\langle M, \Pi - I \rangle \leq 0$.

The result follows by taking the conic combination of these matrices.
\end{proof}

Suppose $f: \bbR^n \rightarrow \bbR$ is a differentiable convex function over a closed convex set $\mathcal{C}$.  Then from first-order optimality one has
$$x^\star \in \arg \min_{x \in \mathcal{C}} f(x) \quad \Leftrightarrow \quad -\nabla f(x^\star) \in \mathcal{N}_{\mathcal{C}}(x^\star).$$
We use this result to prove Proposition \ref{lem:exactness_3_equivalence}.

\begin{proof}[Proof of Proposition \ref{lem:exactness_3_equivalence}]
The proof is similar to the arguments in the proof of Theorem \ref{thm:exactness_indeterminate}.  As before, we define $\tP$ as in \eqref{eq:construction_of_P}, and apply Proposition \ref{thm:tight-closure}.

First, Requirements (i), (ii), and (iii) of Proposition \ref{thm:tight-closure} are satisfied for precisely the same reasons.  It remains to verify Requirement (iv).  Consider the function 
$$
\bx \mapsto \bx^T \tP \bx \quad \bx = \mathrm{vec}(X), X \in \mathrm{DS}(n).
$$
Since $\tP \in \mathcal{S}$, the function is convex over $\mathrm{DS}(n)$.  The derivative of this function is $2 \tP \bx$.  In particular, the (unvectorized) derivative when evaluated at $\bx = \mathrm{vec}(I)$ is
$$ 
\mathrm{mat}(2 \tP \bx) = 2\sum_{i,j} \bu^{(ij)} (\be_i+\be_{j})^T + 2(\be_{i}+\be_{j}) \bv^{(ij)T}. 
$$
By combining the assumption and Lemma \ref{lem:normalcone-description}, we have $-\mathrm{mat}(2 \tP \bx) \in \mathcal{N}_{\mathrm{DS}(n)}(I)$, from which we conclude our result.
\end{proof}

\subsection{Numerical Experiment on Adjacency Matrices of Graphs}

In this experiment, the matrices $A$ and $-B$ are adjacency matrices of graphs.  Our numerical experiments can be viewed as direct implementations of the graph matching problem.  Part of the purpose of this experiment is to investigate the efficacy of \eqref{eq:qap-sdp} for graph matching instances, at least over small dimensions.

More specifically, we take $A$ and $-B$ to be all possible pairs of (adjacency matrices of) non-isomorphic graphs over $n$ vertices, where $n \in \{3,4,5,6\}$.  We solve the associated QAP instances by testing all possible $n!$ permutations, and we declare exactness if the QAP objective matches that of the semidefinite relaxation \eqref{eq:qap-sdp}.  
For an instance where the relaxation \eqref{eq:qap-sdp} is exact, we declare that a certificate from Theorem \ref{thm:exactness_indeterminate} exists if the procedure outlined in Section \ref{sec:find_certificate} succeeds at finding $\{\bu^{(ij)}\}$ and $\{\bv^{(kl)}\}$.

We present the result of our experiments in Table \ref{tab:percentage_exactness_graph}.

In Appendix \ref{append:graph}, we detail the results for $n=5$: In Figure \ref{fig:graphs_n5} we list all possible non-isomorphic graphs over five vertices, and we present the optimal values of \eqref{eq:qap-sdp} for every pair of matrices in Table \ref{tab:graph_n5_opt_val}.

\begin{table}[htbp]
    \centering
    \begin{tabular}{|c||cccc|}
    \hline
        n & 3 & 4 & 5 & 6 \\
    \hline
        $\#$ of non-isomorphic graphs & 3 & 11 & 34 & 154 \\
        $\#$ of graphs pairs & 6 & 66 & 595 & 12,246 \\ 
        $\#$ of exact graph pairs & 6 & 66 & 595 & 11,572 \\
        $\#$ of graph pairs with certificates & 6 & 66 & 595 & 11,480 \\
        Percentage of exactness ($\%$) & 100 & 100 & 100 & 94.5 \\
        Percentage of graph pairs with certificates ($\%$) & 100 & 100 & 100 & 99.2 \\
    \hline
    \end{tabular}
    \caption{Frequency of exactness of \eqref{eq:qap-sdp} on graph matching and the percentage of instances for a certificate from Theorem \ref{thm:exactness_indeterminate} exists. The matrices $A$ and $-B$ are adjacency matrices of graphs. }
    \label{tab:percentage_exactness_graph}
\end{table}

Our results indicate that \eqref{eq:qap-sdp} is exact for graph matching tasks involving up to $5$ vertices.  In addition, a certificate of exactness via Theorem \ref{thm:exactness_indeterminate} exists for a large proportion of these instances.  For $n=6$ vertices, there are pairs of matrices for which the semidefinite relaxation \eqref{eq:qap-sdp} is not exact.  Here is one such example.  Take:
$$
A = \left( \begin{array}{cccccc}
    0 & 1 & 0 & 0 & 0 & 0 \\
    1 & 0 & 0 & 0 & 0 & 0 \\
    0 & 0 & 0 & 1 & 0 & 0 \\
    0 & 0 & 1 & 0 & 0 & 0 \\
    0 & 0 & 0 & 0 & 0 & 1 \\
    0 & 0 & 0 & 0 & 1 & 0
\end{array}\right), \quad
B = - \left( \begin{array}{cccccc}
    0 & 1 & 1 & 0 & 0 & 0 \\
    1 & 0 & 1 & 0 & 0 & 0 \\
    1 & 1 & 0 & 0 & 0 & 0 \\
    0 & 0 & 0 & 0 & 0 & 0 \\
    0 & 0 & 0 & 0 & 0 & 0 \\
    0 & 0 & 0 & 0 & 0 & 0
\end{array}\right). 
$$
The adjacency matrices $A$ and $-B$ in this example correspond to the graphs pair in Figure \ref{fig:inexact_example}. For every possible mapping between the vertices of these two graphs, the number of overlapping edges is either zero or one. Therefore, the minimum of \eqref{eq:qap} is achieved for any mapping where the number of overlapping edges is exactly one, which results in an optimal value of $-4$.  To see why \eqref{eq:qap-sdp} is not exact, let $\hat{X} = E/6$, and define $\hat{\tX}$ as the matrix below:
$$
\hat{\tX} = \frac{1}{48} \left( \begin{array}{cccccc}
    \hat{Y}_1 & \hat{Y}_2 & \hat{Y}_2 & \hat{Y}_3 & \hat{Y}_3 & \hat{Y}_3 \\
    \hat{Y}_2 & \hat{Y}_1 & \hat{Y}_2 & \hat{Y}_3 & \hat{Y}_3 & \hat{Y}_3 \\
    \hat{Y}_2 & \hat{Y}_2 & \hat{Y}_1 & \hat{Y}_3 & \hat{Y}_3 & \hat{Y}_3 \\
    \hat{Y}_3 & \hat{Y}_3 & \hat{Y}_3 & \hat{Y}_1 & \hat{Y}_2 & \hat{Y}_2 \\
    \hat{Y}_3 & \hat{Y}_3 & \hat{Y}_3 & \hat{Y}_2 & \hat{Y}_1 & \hat{Y}_2 \\
    \hat{Y}_3 & \hat{Y}_3 & \hat{Y}_3 & \hat{Y}_2 & \hat{Y}_2 & \hat{Y}_1 \\
\end{array} \right),
$$
and
$$
\hat{Y}_1 = 8I, 
\hat{Y}_2 = \left( \begin{array}{cccccc}
    0 & 4 & 1 & 1 & 1 & 1 \\
    4 & 0 & 1 & 1 & 1 & 1 \\
    1 & 1 & 0 & 4 & 1 & 1 \\
    1 & 1 & 4 & 0 & 1 & 1 \\
    1 & 1 & 1 & 1 & 0 & 4 \\
    1 & 1 & 1 & 1 & 4 & 0 \\
\end{array} \right), 
\hat{Y}_3 = \left( \begin{array}{cccccc}
    0 & 0 & 2 & 2 & 2 & 2 \\
    0 & 0 & 2 & 2 & 2 & 2 \\
    2 & 2 & 0 & 0 & 2 & 2 \\
    2 & 2 & 0 & 0 & 2 & 2 \\
    2 & 2 & 2 & 2 & 0 & 0 \\
    2 & 2 & 2 & 2 & 0 & 0 \\
\end{array} \right). 
$$
It can be verified that the pair $(\hat{X},\hat{\tX})$ is a feasible solution to \eqref{eq:qap-sdp}. However, the value of $\langle A \otimes B, \hat{\tX} \rangle$ is $-8$, which is strictly smaller than the true optimal value. Consequently, the optimal value of \eqref{eq:qap-sdp} is strictly smaller than the true optimal value, which tells us that the semidefinite relaxation is not exact in this instance.

\begin{figure}[htbp]
    \centering
    \includegraphics[width=0.5\linewidth]{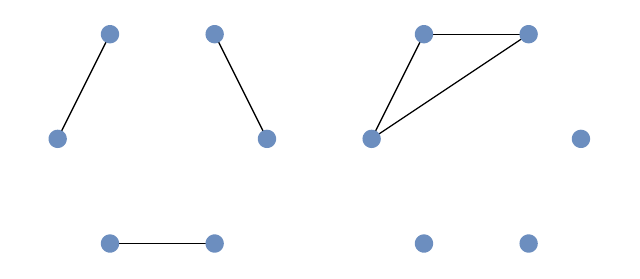}
    \caption{Graphs correspond to $A$ (left) and $-B$ (right). }
    \label{fig:inexact_example}
\end{figure}

Through this numerical experiment, we found an instance of a pair of matrices $A$ and $B$ of dimension $6$ where the semidefinite relaxation \eqref{eq:qap-sdp}.  It leaves the question -- what is the largest dimension for which the relaxation is always exact?  Theorem \ref{thm:exactness_n3} tells us this is at least $3$.  It would be interesting to work out what the exact number is.



\subsection{Numerical Experiment on Distances Matrices}  

In this experiment, the matrices $A$ and $-B$ are the distances matrices of a pair of point clouds.  The matrices are generated as follows.  First, we generate two sets of samples points of size $n$ each whereby $\mathbf{S}_c$ are $n$ i.i.d. draws from the Gaussian distribution $\mathcal{N}(0,I_3)$, while $\mathbf{S}_d$ is a different set of $n$ i.i.d. draws from the distribution $\mathcal{N}(\mu_d ,I_3)$, $\mu = [4,4,4]^T$.  See Figure \ref{fig:samples} for an illustration of the generated sets of sample points.

\begin{figure}[htbp]
    \centering
\includegraphics[width=0.5\linewidth]{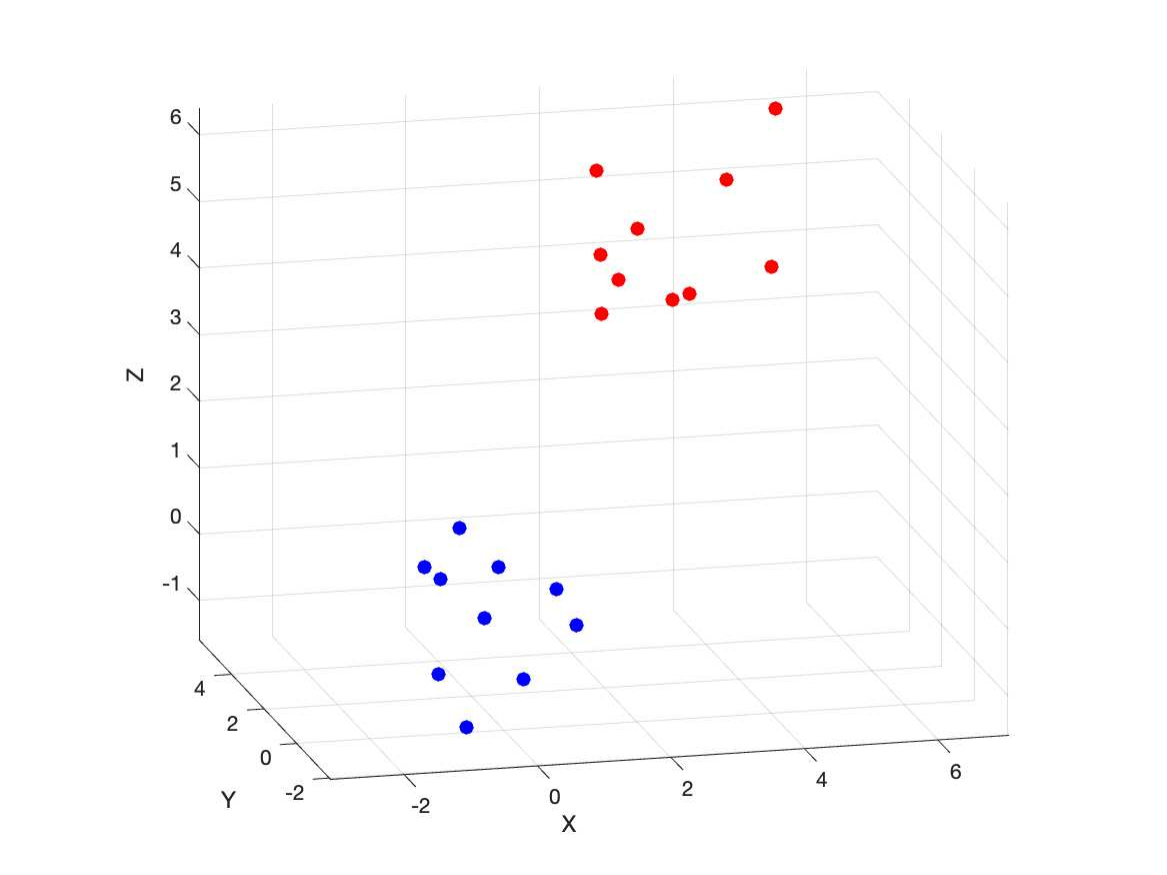}
    \caption{Illustration of the generated sets of sample points $\mathbf{S}_c$ (in blue) and $\mathbf{S}_d$ (in red) for $n=10$. }
    \label{fig:samples}
\end{figure}

Next, we compute the matrix $C \in \mathbb{R}^{n \times n}$ where $C_{ij}$ is the Euclidean distance between the $i$-th and $j$-th points in $\mathbf{S}_c$.  The matrix $D$ is similarly computed.  Last, we set $A = C$ and $B = -D$.  For each value of $n \in \{3,\ldots,10\}$, we perform $100$ random trials based on the above procedure.  We declare that the semidefinite relaxation is exact if the optimal solution to \eqref{eq:qap-sdp} is a permutation matrix.  We say that a certificate from Theorem \ref{thm:exactness_indeterminate} exists if the procedure outlined in Section \ref{sec:find_certificate} succeeds at finding $\{\bu^{(ij)}\}$ and $\{\bv^{(kl)}\}$ satisfying the conditions specified in Theorem \ref{thm:exactness_indeterminate}.  

We present the results in Table \ref{tab:certification_distance}. 
\begin{table}[htbp]
    \centering
    \begin{tabular}{|c||cccccccc|}
    \hline
       n & 3 & 4 & 5 & 6 & 7 & 8 & 9 & 10 \\
       \hline
       Percentage of exact pairs ($\%$) & 100 & 100 & 100 & 100 & 100 & 100 & 100 & 99\\
       Percentage of pairs with certificates ($\%$) & 100 & 100 & 100 & 98 & 99 & 93 & 96 & 91 \\
       \hline
    \end{tabular}
    \caption{The percentage of instances where \eqref{eq:qap-sdp} is exact and the percentage of instances for a certificate from Theorem \ref{thm:exactness_indeterminate} exists. The matrices $A$ and $-B$ are generated from random distance matrices. }
    \label{tab:certification_distance}
\end{table}

Our results show that a large proportion of the simulated instances have a certificate from Theorem \ref{thm:exactness_indeterminate}.  We do note that the proportion of instances for which there is a certificate drops as we increase $n$.  It is not so clear how one compares QAP instances from different dimensions, but a plausible explanation as to why the percentage of exactness as well as instances where a certificate for Theorem \ref{thm:exactness_indeterminate} drops as we increase $n$ is that -- at least at an intuitive level -- it becomes increasingly more difficult to obtain a permutation as an optimal solution because we need a larger number of variables to be integral.

We also note that the percentage of instances for which the relaxation \eqref{eq:qap-sdp} is exact is substantially higher with distance matrices.  This suggests the possibility that the relaxation \eqref{eq:qap-sdp} may perform better on {\em structured} data $A$ and $B$, such as those arising from actual distances or metrics, as opposed to completely generic matrices $A$ and $B$.  





\bibliography{biblio}

\begin{thebibliography}{10}

\bibitem{aflalo2015convex}
Yonathan Aflalo, Alexander Bronstein, and Ron Kimmel.
\newblock On {C}onvex {R}elaxation of {G}raph {I}somorphism.
\newblock {\em Proceedings of the National Academy of Sciences}, 112(10):2942--2947, 2015.

\bibitem{qapsolution}
Kurt~M Anstreicher.
\newblock Recent {A}dvances in the {S}olution of {Q}uadratic {A}ssignment {P}roblems.
\newblock {\em Mathematical Programming}, 97:27--42, 2003.

\bibitem{mosek}
MOSEK ApS.
\newblock {\em The MOSEK Optimization Toolbox for MATLAB Manual. Version 9.1.9.}, 2024.

\bibitem{relaxclusteringawasthi}
Pranjal Awasthi, Afonso~S. Bandeira, Moses Charikar, Ravishankar Krishnaswamy, Soledad Villar, and Rachel Ward.
\newblock Relax, {N}o {N}eed to {R}ound: {I}ntegrality of {C}lustering {F}ormulations.
\newblock In {\em Proceedings of the 2015 Conference on Innovations in Theoretical Computer Science}, ITCS '15, page 191–200, New York, NY, USA, 2015. Association for Computing Machinery.

\bibitem{bandeira2018randomsychro}
Afonso~S Bandeira.
\newblock Random {L}aplacian {M}atrices and {C}onvex {R}elaxations.
\newblock {\em Foundations of Computational Mathematics}, 18:345--379, 2018.

\bibitem{bernard2018ds}
Florian Bernard, Christian Theobalt, and Michael Moeller.
\newblock {DS*}: {T}ighter {L}ifting-free {C}onvex {R}elaxations for {Q}uadratic {M}atching {P}roblems.
\newblock In {\em Proceedings of the IEEE Conference on Computer Vision and Pattern Recognition}, pages 4310--4319, 2018.

\bibitem{bravo2018semidefinite}
Jos{\'e} F.~S. Bravo~Ferreira, Yuehaw Khoo, and Amit Singer.
\newblock Semidefinite {P}rogramming {A}pproach for the {Q}uadratic {A}ssignment {P}roblem with a {S}parse {G}raph.
\newblock {\em Computational Optimization and Applications}, 69:677--712, 2018.

\bibitem{burer2020exact}
Samuel Burer and Yinyu Ye.
\newblock Exact {S}emidefinite {F}ormulations for a {C}lass of ({R}andom and {N}on-random) {N}onconvex {Q}uadratic {P}rograms.
\newblock {\em Mathematical Programming}, 181(1):1--17, 2020.

\bibitem{svt}
Jian-Feng Cai, Emmanuel~J. Cand\`{e}s, and Zuowei Shen.
\newblock A {S}ingular {V}alue {T}hresholding {A}lgorithm for {M}atrix {C}ompletion.
\newblock {\em SIAM Journal on Optimization}, 20(4):1956--1982, 2010.

\bibitem{exactmccandesfocm}
Emmanuel Cand\`{e}s and Benjamin Recht.
\newblock Exact {M}atrix {C}ompletion via {C}onvex {O}ptimization.
\newblock {\em Foundations of Computational Mathematics}, 9:717–772, 2009.

\bibitem{exactmccandesacm}
Emmanuel Cand\`{e}s and Benjamin Recht.
\newblock Exact {M}atrix {C}ompletion via {C}onvex {O}ptimization.
\newblock {\em Communications of the ACM}, 55(6):111--119, 2012.

\bibitem{candes2015phase}
Emmanuel~J. Cand\`{e}s, Yonina~C. Eldar, Thomas Strohmer, and Vladislav Voroninski.
\newblock Phase {R}etrieval via {M}atrix {C}ompletion.
\newblock {\em SIAM Review}, 57(2):225--251, 2015.

\bibitem{qapbookcela}
Eranda \c{C}ela.
\newblock {\em The {Q}uadratic {A}ssignment {P}roblem: {T}heory and {A}lgorithms}, volume~1.
\newblock Springer Science \& Business Media, 2013.

\bibitem{de2012improvedsdp}
Etienne de~Klerk and Renata Sotirov.
\newblock Improved {S}emidefinite {P}rogramming {B}ounds for {Q}uadratic {A}ssignment {P}roblems with {S}uitable {S}ymmetry.
\newblock {\em Mathematical Programming}, 133:75--91, 2012.

\bibitem{dym2017exact}
Nadav Dym and Yaron Lipman.
\newblock Exact {R}ecovery with {S}ymmetries for {P}rocrustes {M}atching.
\newblock {\em SIAM Journal on Optimization}, 27(3):1513--1530, 2017.

\bibitem{ds++}
Nadav Dym, Haggai Maron, and Yaron Lipman.
\newblock {DS}++: {A} {F}lexible, {S}calable and {P}rovably {T}ight {R}elaxation for {M}atching {P}roblems.
\newblock {\em ACM Trans. Graph.}, 36(6), November 2017.

\bibitem{spectralGMgaussian}
Zhou Fan, Cheng Mao, Yihong Wu, and Jiaming Xu.
\newblock Spectral {G}raph {M}atching and {R}egularized {Q}uadratic {R}elaxations {I}: {A}lgorithm and {G}aussian {A}nalysis.
\newblock {\em Foundations of Computational Mathematics}, 23(5):1511--1565, 2023.

\bibitem{spectralGMerdos}
Zhou Fan, Cheng Mao, Yihong Wu, and Jiaming Xu.
\newblock Spectral {G}raph {M}atching and {R}egularized {Q}uadratic {R}elaxations {II}: Erd{\H{o}}s-{R}{\'e}nyi {G}raphs and {U}niversality.
\newblock {\em Foundations of Computational Mathematics}, 23(5):1567--1617, 2023.

\bibitem{cvx2}
Michael Grant and Stephen Boyd.
\newblock Graph {I}mplementations for {N}onsmooth {C}onvex {P}rograms.
\newblock In V.~Blondel, S.~Boyd, and H.~Kimura, editors, {\em Recent Advances in Learning and Control}, Lecture Notes in Control and Information Sciences, pages 95--110. Springer-Verlag Limited, 2008.

\bibitem{cvx1}
Michael Grant and Stephen Boyd.
\newblock {CVX}: {M}atlab {S}oftware for {D}isciplined {C}onvex {P}rogramming, version 2.1, March 2014.

\bibitem{partialrecoveryhall}
Georgina Hall and Laurent Massouli{\'e}.
\newblock Partial {R}ecovery in the {G}raph {A}lignment {P}roblem.
\newblock {\em Operations Research}, 71(1):259--272, 2023.

\bibitem{lovaszthetagraphcoloring}
Jiaxin Hou, Yong~Sheng Soh, and Antonios Varvitsiotis.
\newblock The {L}ov\'asz {T}heta {F}unction for {R}ecovering {P}lanted {C}lique {C}overs and {G}raph {C}olorings.
\newblock {\em arXiv preprint arXiv:2310.00257}, 2023.

\bibitem{certifiablekmeans}
Takayuki Iguchi, Dustin~G Mixon, Jesse Peterson, and Soledad Villar.
\newblock Probably {C}ertifiably {C}orrect k-{M}eans {C}lustering.
\newblock {\em Mathematical Programming}, 165:605--642, 2017.

\bibitem{reducibility}
Richard~M Karp.
\newblock {\em Reducibility among Combinatorial Problems}.
\newblock Springer, 2010.

\bibitem{kezurer2015tight}
Itay Kezurer, Shahar~Z. Kovalsky, Ronen Basri, and Yaron Lipman.
\newblock Tight {R}elaxation of {Q}uadratic {M}atching.
\newblock In {\em Computer Graphics Forum}, volume~34, pages 115--128. Wiley Online Library, 2015.

\bibitem{leordeanu2005spectral}
Marius Leordeanu and Martial Hebert.
\newblock A {S}pectral {T}echnique for {C}orrespondence {P}roblems using {P}airwise {C}onstraints.
\newblock In {\em Tenth IEEE International Conference on Computer Vision (ICCV'05) Volume 1}, volume~2, pages 1482--1489. IEEE, 2005.

\bibitem{ling2022improvedsynchro}
Shuyang Ling.
\newblock Improved {P}erformance {G}uarantees for {O}rthogonal {G}roup {S}ynchronization via {G}eneralized {P}ower {M}ethod.
\newblock {\em SIAM Journal on Optimization}, 32(2):1018--1048, 2022.

\bibitem{ling2024exactnessqapsdp}
Shuyang Ling.
\newblock On the {E}xactness of {SDP} {R}elaxation for {Q}uadratic {A}ssignment {P}roblem.
\newblock {\em arXiv preprint arXiv:2408.05942}, 2024.

\bibitem{certifyspectralclustering}
Shuyang Ling and Thomas Strohmer.
\newblock Certifying {G}lobal {O}ptimality of {G}raph {C}uts via {S}emidefinite {R}elaxation: A {P}erformance {G}uarantee for {S}pectral {C}lustering.
\newblock {\em Foundations of Computational Mathematics}, 20(3):367--421, 2020.

\bibitem{qapsurvey}
Eliane~Maria Loiola, Nair Maria~Maia De~Abreu, Paulo~Oswaldo Boaventura-Netto, Peter Hahn, and Tania Querido.
\newblock A {S}urvey for the {Q}uadratic {A}ssignment {P}roblem.
\newblock {\em European Journal of Operational Research}, 176(2):657--690, 2007.

\bibitem{lyzinski2015graph}
Vince Lyzinski, Donniell~E Fishkind, Marcelo Fiori, Joshua~T Vogelstein, Carey~E Priebe, and Guillermo Sapiro.
\newblock Graph {M}atching: {R}elax at {Y}our {O}wn {R}isk.
\newblock {\em IEEE Transactions on Pattern Analysis and Machine Intelligence}, 38(1):60--73, 2015.

\bibitem{randomGMimproved}
Cheng Mao, Mark Rudelson, and Konstantin Tikhomirov.
\newblock Random {G}raph {M}atching with {I}mproved {N}oise {R}obustness.
\newblock In {\em Conference on Learning Theory}, pages 3296--3329. PMLR, 2021.

\bibitem{exactrandomGM}
Cheng Mao, Mark Rudelson, and Konstantin Tikhomirov.
\newblock Exact {M}atching of {R}andom {G}raphs with {C}onstant {C}orrelation.
\newblock {\em Probability Theory and Related Fields}, 186(1):327--389, 2023.

\bibitem{Netzer:10}
Tim Netzer.
\newblock On {S}emidefinite {R}epresentations of {N}on-closed {S}ets.
\newblock {\em Linear Algebra and its Applications}, 432(12):3072--3078, 2010.

\bibitem{admmqapsdp}
Danilo~Elias Oliveira, Henry Wolkowicz, and Yangyang Xu.
\newblock {ADMM} for the {SDP} {R}elaxation of the {QAP}.
\newblock {\em Mathematical Programming Computation}, 10(4):631--658, 2018.

\bibitem{povh2009copositive}
Janez Povh and Franz Rendl.
\newblock Copositive and {S}emidefinite {R}elaxations of the {Q}uadratic {A}ssignment {P}roblem.
\newblock {\em Discrete Optimization}, 6(3):231--241, 2009.

\bibitem{rujeerapaiboon2019size}
Napat Rujeerapaiboon, Kilian Schindler, Daniel Kuhn, and Wolfram Wiesemann.
\newblock Size {M}atters: {C}ardinality-{C}onstrained {C}lustering and {O}utlier {D}etection via {C}onic {O}ptimization.
\newblock {\em SIAM Journal on Optimization}, 29(2):1211--1239, 2019.

\bibitem{qapnphard}
Sartaj Sahni and Teofilo Gonzalez.
\newblock P-complete {A}pproximation {P}roblems.
\newblock {\em Journal of the ACM (JACM)}, 23(3):555--565, 1976.

\bibitem{shorrelax}
Naum~Z. Shor.
\newblock Quadratic {O}ptimization {P}roblems.
\newblock {\em Soviet Journal of Computer and Systems Sciences}, 25(6):1--11, 1987.

\bibitem{shapecorr}
Oliver Van~Kaick, Hao Zhang, Ghassan Hamarneh, and Daniel Cohen-Or.
\newblock A {S}urvey on {S}hape {C}orrespondence.
\newblock In {\em Computer Graphics Forum}, volume~30, pages 1681--1707. Wiley Online Library, 2011.

\bibitem{wang2022tightness}
Alex~L. Wang and Fatma K{\i}l{\i}n{\c{c}}-Karzan.
\newblock On the {T}ightness of {SDP} {R}elaxations of {QCQPs}.
\newblock {\em Mathematical Programming}, 193(1):33--73, 2022.

\bibitem{zhao1998semidefinite}
Qing Zhao, Stefan~E. Karisch, Franz Rendl, and Henry Wolkowicz.
\newblock Semidefinite {P}rogramming {R}elaxations for the {Q}uadratic {A}ssignment {P}roblem.
\newblock {\em Journal of Combinatorial Optimization}, 2:71--109, 1998.

\bibitem{zhong2018nearoptsynchro}
Yiqiao Zhong and Nicolas Boumal.
\newblock Near-optimal {B}ounds for {P}hase {S}ynchronization.
\newblock {\em SIAM Journal on Optimization}, 28(2):989--1016, 2018.

\end{thebibliography}

\appendix

\newpage

\section{Redundancy of the Arrow Condition}\label{append:redundancy}

In this section we prove the redundancy of the arrow condition in \eqref{eq:qap-sdp}.  The key step is to show that the constraints in \eqref{eq:qap-sdp} imply the following set of constraints:
\begin{equation}\label{eq:mariginal-redundant}
\sum_{i} \tX_{(i,j),(k,l)} = X_{kl} \text{ and } \sum_{i} \tX_{(j,i),(k,l)} = X_{kl} \text{ for all $j,k,l \in [n]$}. 
\end{equation}
Then one has $\sum_i \tX_{(i,j),(i,l)} = X_{il}$.  Recall that $\tX_{(i,j),(i,l)} = 0$ if $j \neq l$.  By combining these conclusions, one has $\tX_{(i,l),(i,l)} = X_{il}$, which implies that the arrow condition is redundant.

We proceed to show why \eqref{eq:mariginal-redundant} holds.  Given a quadratic function in $X$ of the form
$$
g(X) = \mathrm{vec}(X)^T \tA \mathrm{vec}(X) + \langle X, B \rangle + c,
$$
we define the linearization $\bar{g}$ obtained by replacing all quadratic terms with terms in $\tX$:
$$
\bar{g} (X) = \langle \tA, \tX \rangle + \langle X, B \rangle + c,
$$

The following result is based on Proposition 4 in \cite{dym2017exact}.  Specifically, it is based on a restatement of the result in \cite{dym2017exact} that can be found in the Appendix of \cite{ds++}.  In the following, we state the result specialized to our setting.

\begin{proposition}[Proposition 4, \cite{dym2017exact}]\label{lem:quadratic-lin-ds++}
Let $p = \langle A_p, X \rangle + b_p$ be an affine function in $X$, and let $g = p^2$.  Suppose $\bar{g}(X,\tX) = 0$ for all $(X,\tX)$ feasible in \eqref{eq:qap-sdp}.  Then $\bar{f}(X,\tX) = 0$ for all quadratic functions $f$ of the form $f=pq$, where $q = \langle A_q , X \rangle + b_q$ is any affine function in $X$.
\end{proposition}

\begin{proposition} \label{thm:marginalsums}
The constraints in \eqref{eq:qap-sdp} imply \eqref{eq:mariginal-redundant}.
\end{proposition}

\begin{proof}[Proof of Proposition \ref{thm:marginalsums}]
Define
$$
L_j(X) := \Big(\sum_{i}X_{ij}-1\Big)^2 \text{ and } R_j(X) := \Big(\sum_{i}X_{ji}-1\Big)^2. 
$$
Then one has
$$
\bar{L}_j(X,\tX) = \sum_{i,k} \tX_{(i,j),(k,j)} - 2 \sum_{i} X_{ij} + 1. 
$$
Let $(X,\tX)$ be feasible in \eqref{eq:qap-sdp}.  Then one has $\sum_{i}X_{ij} = 1$ and $\sum_{i,k} \tX_{(i,j),(k,j)} \overset{(a)}{=} \sum_{i} \tX_{(i,j),(i,j)} \overset{(b)}{=} 1$, where (a) follows from the Gangster operator and (b) follows from the constraint in the second row of \eqref{eq:qap-sdp}. This implies $\bar{L}_j(X,\tX) = 0$ for all $j$.  By a similar set of arguments, one also has $\bar{R}_j(X,\tX) = 0$.

We are now in a position to apply Proposition \ref{lem:quadratic-lin-ds++} with $p = \sum_{i} X_{ij}-1$ and $p = \sum_{i} X_{ji}-1$.  Consider
$$
l_{j,k,l}(X) := \Big(\sum_{i}X_{ij}-1\Big)X_{kl} \text{ and } r_{j,k,l}(X) := \Big(\sum_{i}X_{ji}-1\Big)X_{kl}.
$$
By Proposition \ref{lem:quadratic-lin-ds++}, we have $\bar{l}_{j,k,l}(X,\tX) = 0$ and $\bar{r}_{j,k,l}(X,\tX) = 0$ for all feasible $(X,\tX)$ and all $j,k,l \in [n]$.  This is precisely \eqref{eq:mariginal-redundant}.
\end{proof}

\newpage

\section{Supplementary Material}\label{append:graph}

\begin{figure}[htbp]
    \centering
    \includegraphics[width=0.8\linewidth]{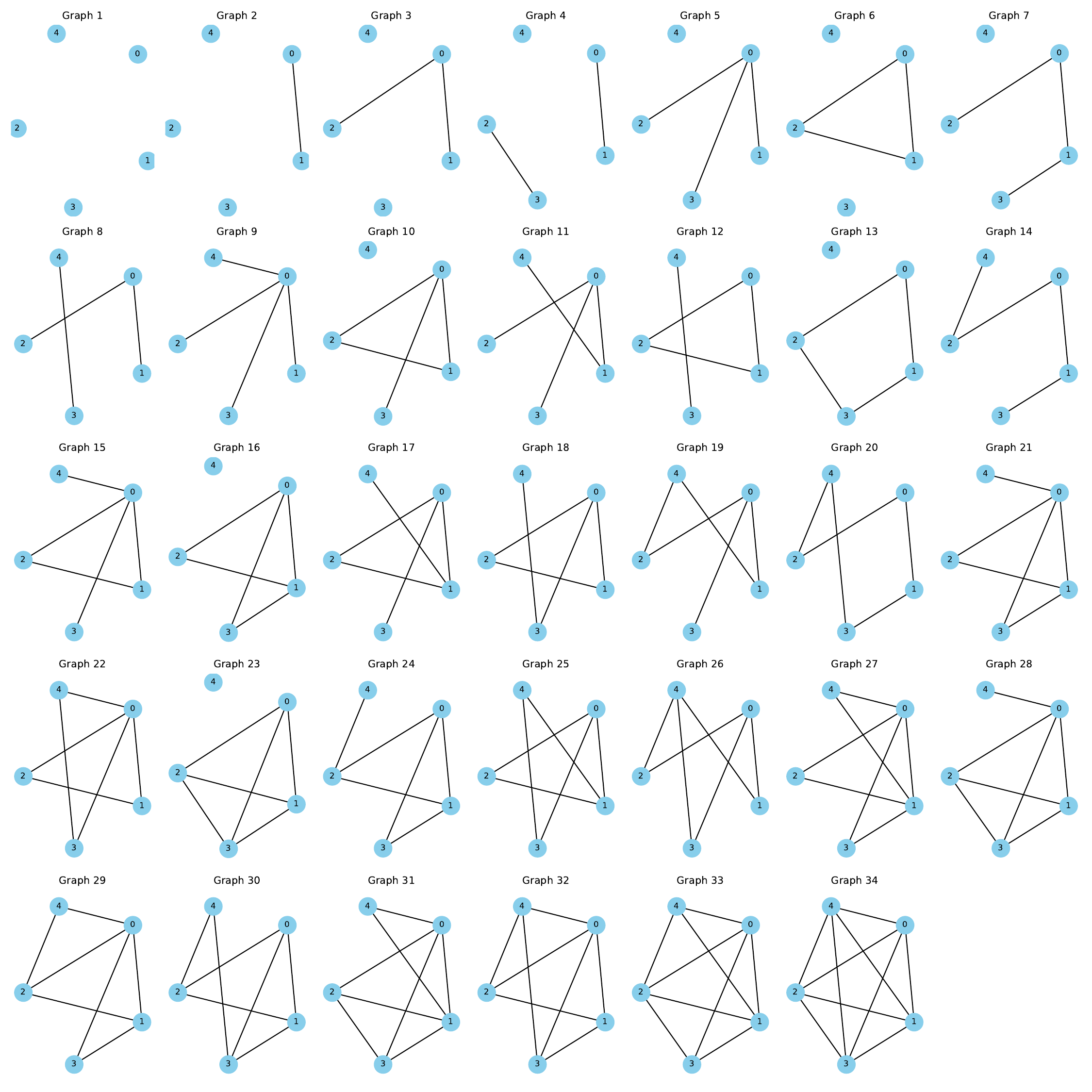}
    \caption{All non-isomorphic graphs with $5$ vertices. }
    \label{fig:graphs_n5}
\end{figure}

\begin{table}[htbp]
\centering
\rotatebox{90}{
\resizebox{\textwidth}{!}{
\begin{tabular}{c||c|c|c|c|c|c|c|c|c|c|c|c|c|c|c|c|c|c|c|c|c|c|c|c|c|c|c|c|c|c|c|c|c|c|}
  & 1 & 2 & 3 & 4 & 5 & 6 & 7 & 8 & 9 & 10 & 11 & 12 & 13 & 14 & 15 & 16 & 17 & 18 & 19 & 20 & 21 & 22 & 23 & 24 & 25 & 26 & 27 & 28 & 29 & 30 & 31 & 32 & 33 & 34 \\
\hline\hline
1 &   & 2 & 4 & 4 & 6 & 6 & 6 & 6 & 8 & 8 & 8 & 8 & 8 & 8 & 10 & 10 & 10 & 10 & 10 & 10 & 12 & 12 & 12 & 12 & 12 & 12 & 14 & 14 & 14 & 14 & 16 & 16 & 18 & 20 \\
\hline
2 &   &  & 2 & 2 & 4 & 4 & 4 & 4 & 6 & 6 & 6 & 6 & 6 & 6 & 8 & 8 & 8 & 8 & 8 & 8 & 10 & 10 & 10 & 10 & 10 & 10 & 12 & 12 & 12 & 12 & 14 & 14 & 16 & 18 \\
\hline
3 &   &  &   & 4 & 2 & 2 & 2 & 2 & 4 & 4 & 4 & 4 & 4 & 4 & 6 & 6 & 6 & 6 & 6 & 6 & 8 & 8 & 8 & 8 & 8 & 8 & 10 & 10 & 10 & 10 & 12 & 12 & 14 & 16 \\
\hline
4 &   &  &   &   & 6 & 6 & 2 & 2 & 8 & 4 & 4 & 4 & 4 & 4 & 6 & 6 & 6 & 6 & 6 & 6 & 8 & 8 & 8 & 8 & 8 & 8 & 10 & 10 & 10 & 10 & 12 & 12 & 14 & 16 \\
\hline
5 &   &  &   &   &   & 4 & 4 & 4 & 2 & 2 & 2 & 6 & 6 & 6 & 4 & 4 & 4 & 4 & 4 & 8 & 6 & 6 & 6 & 6 & 6 & 6 & 8 & 8 & 8 & 8 & 10 & 10 & 12 & 14 \\
\hline
6 &   &  &   &   &   &   & 4 & 4 & 6 & 2 & 6 & 2 & 6 & 6 & 4 & 4 & 4 & 4 & 8 & 8 & 6 & 6 & 6 & 6 & 6 & 10 & 8 & 8 & 8 & 8 & 10 & 10 & 12 & 14 \\
\hline
7 &   &  &   &   &   &   &   & 4 & 6 & 2 & 2 & 6 & 2 & 2 & 4 & 4 & 4 & 4 & 4 & 4 & 6 & 6 & 6 & 6 & 6 & 6 & 8 & 8 & 8 & 8 & 10 & 10 & 12 & 14 \\
\hline
8 &   &  &   &   &   &   &   &   & 6 & 6 & 2 & 2 & 6 & 2 & 4 & 8 & 4 & 4 & 4 & 4 & 6 & 6 & 10 & 6 & 6 & 6 & 8 & 8 & 8 & 8 & 10 & 10 & 12 & 14 \\
\hline
9 &   &  &   &   &   &   &   &   &   & 4 & 4 & 8 & 8 & 8 & 2 & 6 & 6 & 6 & 6 & 10 & 4 & 4 & 8 & 8 & 8 & 8 & 6 & 6 & 6 & 10 & 8 & 8 & 10 & 12 \\
\hline
10 &  &   &   &   &   &   &   &   &   &   & 4 & 4 & 4 & 4 & 2 & 2 & 2 & 2 & 6 & 6 & 4 & 4 & 4 & 4 & 4 & 8 & 6 & 6 & 6 & 6 & 8 & 8 & 10 & 12 \\
\hline
11 &  &   &   &   &   &   &   &   &   &   &   & 4 & 4 & 4 & 2 & 6 & 2 & 2 & 2 & 6 & 4 & 4 & 8 & 4 & 4 & 4 & 6 & 6 & 6 & 6 & 8 & 8 & 10 & 12 \\
\hline
12 &  &   &   &   &   &   &   &   &   &   &   &   & 8 & 4 & 6 & 6 & 6 & 2 & 6 & 6 & 8 & 4 & 8 & 4 & 4 & 8 & 10 & 6 & 6 & 6 & 8 & 8 & 10 & 12 \\
\hline
13 &  &   &   &   &   &   &   &   &   &   &   &   &   & 4 & 6 & 2 & 6 & 6 & 2 & 6 & 4 & 8 & 4 & 4 & 4 & 4 & 6 & 6 & 6 & 6 & 8 & 8 & 10 & 12 \\
\hline
14 &  &   &   &   &   &   &   &   &   &   &   &   &   &   & 6 & 6 & 2 & 2 & 2 & 2 & 4 & 4 & 8 & 4 & 4 & 4 & 6 & 6 & 6 & 6 & 8 & 8 & 10 & 12 \\
\hline
15 &  &   &   &   &   &   &   &   &   &   &   &   &   &   &   & 4 & 4 & 4 & 4 & 8 & 2 & 2 & 6 & 6 & 6 & 6 & 4 & 4 & 4 & 8 & 6 & 6 & 8 & 10 \\
\hline
16 &  &   &   &   &   &   &   &   &   &   &   &   &   &   &   &   & 4 & 4 & 4 & 8 & 2 & 6 & 2 & 2 & 6 & 6 & 4 & 4 & 4 & 4 & 6 & 6 & 8 & 10 \\
\hline
17 &  &   &   &   &   &   &   &   &   &   &   &   &   &   &   &   &   & 4 & 4 & 4 & 2 & 6 & 6 & 2 & 2 & 6 & 4 & 4 & 4 & 4 & 6 & 6 & 8 & 10 \\
\hline
18 &  &   &   &   &   &   &   &   &   &   &   &   &   &   &   &   &   &   & 4 & 4 & 6 & 2 & 6 & 2 & 2 & 6 & 8 & 4 & 4 & 4 & 6 & 6 & 8 & 10 \\
\hline
19 &  &   &   &   &   &   &   &   &   &   &   &   &   &   &   &   &   &   &   & 4 & 2 & 6 & 6 & 2 & 2 & 2 & 4 & 4 & 4 & 4 & 6 & 6 & 8 & 10 \\
\hline
20 &  &   &   &   &   &   &   &   &   &   &   &   &   &   &   &   &   &   &   &   & 6 & 6 & 10 & 6 & 2 & 6 & 8 & 8 & 4 & 4 & 6 & 6 & 8 & 10 \\
\hline
21 &   &  &   &   &   &   &   &   &   &   &   &   &   &   &   &   &   &   &   &   &   & 4 & 4 & 4 & 4 & 4 & 2 & 2 & 2 & 6 & 4 & 4 & 6 & 8 \\
\hline
22 &  &   &   &   &   &   &   &   &   &   &   &   &   &   &   &   &   &   &   &   &   &   & 8 & 4 & 4 & 8 & 6 & 6 & 2 & 6 & 4 & 4 & 6 & 8 \\
\hline
23 &  &   &   &   &   &   &   &   &   &   &   &   &   &   &   &   &   &   &   &   &   &   &   & 4 & 8 & 8 & 6 & 2 & 6 & 6 & 4 & 8 & 6 & 8 \\
\hline
24 &  &   &   &   &   &   &   &   &   &   &   &   &   &   &   &   &   &   &   &   &   &   &   &   & 4 & 4 & 6 & 2 & 2 & 2 & 4 & 4 & 6 & 8 \\
\hline
25 &  &   &   &   &   &   &   &   &   &   &   &   &   &   &   &   &   &   &   &   &   &   &   &   &   & 4 & 6 & 6 & 2 & 2 & 4 & 4 & 6 & 8 \\
\hline
26 &  &   &   &   &   &   &   &   &   &   &   &   &   &   &   &   &   &   &   &   &   &   &   &   &   &   & 2 & 6 & 6 & 2 & 4 & 4 & 6 & 8 \\
\hline
27 &  &   &   &   &   &   &   &   &   &   &   &   &   &   &   &   &   &   &   &   &   &   &   &   &   &   &   & 4 & 4 & 4 & 2 & 6 & 4 & 6 \\
\hline
28 &  &   &   &   &   &   &   &   &   &   &   &   &   &   &   &   &   &   &   &   &   &   &   &   &   &   &   &   & 4 & 4 & 2 & 6 & 4 & 6 \\
\hline
29 &  &   &   &   &   &   &   &   &   &   &   &   &   &   &   &   &   &   &   &   &   &   &   &   &   &   &   &   &  & 4 & 2 & 2 & 4 & 6 \\
\hline
30 &  &   &   &   &   &   &   &   &   &   &   &   &   &   &   &   &   &   &   &   &   &   &   &   &   &   &   &   &   &  & 2 & 2 & 4 & 6 \\
\hline
31 &  &   &   &   &   &   &   &   &   &   &   &   &   &   &   &   &   &   &   &   &   &   &   &   &   &   &   &   &   &   &  & 4 & 2 & 4 \\
\hline
32 &  &   &   &   &   &   &   &   &   &   &   &   &   &   &   &   &   &   &   &   &   &   &   &   &   &   &   &   &   &   &   &  & 2 & 4 \\
\hline
33 &  &   &   &   &   &   &   &   &   &   &   &   &   &   &   &   &   &   &   &   &   &   &   &   &   &   &   &   &   &   &   &  &  & 2 \\
\hline
34 &  &   &   &   &   &   &   &   &   &   &   &   &   &   &   &   &   &   &   &   &   &   &   &   &   &   &   &   &   &   &   &   &   &  \\
\hline
\end{tabular}
}
}
\caption{Performance of \eqref{eq:qap-sdp} in matching graphs with $5$ vertices. Each entry $(i,j)$ in the table represents the value of $\|\hat{X}C\hat{X}^T-D\|^2_F$, where $C$ and $D$ are adjacency matrices of graph $i$ and graph $j$ in Figure \ref{fig:graphs_n5} and $\hat{X}$ is the optimal solution to \eqref{eq:qap-sdp} with $A = C$ and $B = -D$. For clarity, zero entries in the diagonal and entries in the lower triangular portion of the table are omitted (due to symmetry).}
\label{tab:graph_n5_opt_val}
\end{table}

\end{document}